\documentclass[11pt]{amsart}

\usepackage{amssymb}
\usepackage{fullpage}

\newtheorem{theorem}{Theorem}[section]
\newtheorem{lemma}[theorem]{Lemma}

\newtheorem{proposition}[theorem]{Proposition}

\newtheorem{remark}[theorem]{Remark}

\numberwithin{equation}{section}
\newcommand{\Ai}{\text{Ai\,}}

\newcommand{\re}{\text{Re\,}}

\newcommand{\pii}{2\pi\mathrm{i}}

\begin{document}
\setcounter{page}{1}

\thanks{Supported by the grant KAW 2015.0270 from the Knut and Alice Wallenberg
Foundation}

\title[The two-time distribution in geometric last-passage percolation]
{The two-time distribution in geometric last-passage percolation}
\author[K.~Johansson]{Kurt Johansson}

\address{
Department of Mathematics,
KTH Royal Institute of Technology,
SE-100 44 Stockholm, Sweden}

\email{kurtj@kth.se}

\begin{abstract}
We study the two-time distribution in directed last passage percolation with geometric weights in the first quadrant. We compute the scaling limit and show that
it is given by a contour integral of a Fredholm determinant.

\end{abstract}

\maketitle

\section{Introduction}\label{secintro}

In this paper we will consider the so called two-time distribution in directed last-passage percolation with geometric weights. This last-passage percolation model has several interpretations. It can be related to the Totally Asymmetric Simple Exclusion Process (TASEP) and to local random growth models. It is a basic example of a solvable model in the KPZ universality class. It has been less clear to what extent the two-time problem is also solvable but recently there has been some developments in this direction, \cite{Dots1}, \cite{JoTt}, \cite{FerSpo}, \cite{NarDou}, \cite{Dous} and \cite{BaiLiu}.
The approach in this paper is different in many ways from that in our previous work \cite{JoTt}. It is closer to standard computations for determinantal processes, more straightforward and simpler. 

To define the model, let $\left(w(i,j)\right)_{i,j\ge 1}$ be independent geometric random variables with parameter $q$,
$$
\mathbb{P}[w(i,j)=k]=(1-q)q^k,\quad k\ge 0.
$$
Consider the last-passage times
\begin{equation}\label{gmn}
G(m,n)=\max_{\pi:(1,1)\nearrow (m,n)} \sum_{(i,j)\in\pi} w(i,j),
\end{equation}
where the maximum is over all up/right paths from $(1,1)$ to $(m,n)$, see \cite{JoSh}. We are interested in the correlation between $G(m_1,n_1)$ and $G(m_2,n_2)$, when
$(m_1,n_1)$ and $(m_2,n_2)$ are ordered in the time-like direction, i.e. $m_1<m_2$ and $n_1<n_2$. To see why this is called a time-like direction, and give one reason why we are interested in the two-time problem, let us reinterpret the model as a discrete polynuclear growth model. It is clear from (\ref{gmn}) that
\begin{equation}\label{growth}
G(m,n)=\max(G(m-1,n), G(m,n-1))+w(m.n).
\end{equation}
Let $G(m,n)=0$ if $(m,n)\notin\mathbb{Z}_+^2$, and define the height function $h(x,t)$ by
\begin{equation}\label{hxt}
h(x,t)=G\left(\frac{t+x+1}2,\frac{t-x+1}2\right)
\end{equation}
for $x+t$ odd, and extend it to all $x\in\mathbb{R}$ by linear interpolation. Then (\ref{growth}) leads to a growth rule for $h(x,t)$ and this is the discrete time and space polynuclear growth model. We think of $x\mapsto h(x,t)$ as the height above $x$ at time $t$, and we get a random one-dimensional interface. Let the constants $c_i$ be given by (\ref{scalingconstants}).
It is known, see \cite{JoDPG}, that the rescaled process
\begin{equation}\label{hrescaled}
\mathcal{H}_T(\eta,t)=\frac{h(2c_1\eta(tT)^{2/3},2tT)-c_2tT}{c_3(tT)^{1/3}},
\end{equation}
as a process in $\eta\in\mathbb{R}$ for a fixed $t>0$, converges as $T\to\infty$ to $\mathcal{A}_2(\eta)-\eta^2$, where $\mathcal{A}_2(\eta)$ is the Airy-2-process, \cite{PrSp}. In particular,
for any fixed $\eta,t$, 
\begin{equation*}
\lim_{T\to\infty}\mathbb{P}[\mathcal{H}_T(\eta,t)\le\xi-\eta^2]=F_2(\xi)=\det(I-K_{\Ai})_{L^2(\xi,\infty)},
\end {equation*}
where $F_2$ is the Tracy-Widom distribution, and
\begin{equation*}
K_{\text{Ai}}(x,y)=\int_0^\infty \Ai(x+s)\Ai(y+s)\,ds,
\end{equation*}
is the Airy kernel. The two-time problem is concerned with the question of the correlation between heights at different times. What is the limiting joint distribution of 
$\mathcal{H}_T(\eta_1,t_1)$ and $\mathcal{H}_T(\eta_2,t_2)$ for $t_1<t_2$, as $T\to\infty$? From (\ref{hxt}), we see that this is related to understanding the correlation between last-passage times in the time-like direction. That a time separation of order $T$ is the correct order to get non-trivial correlations is quite clear if we think about how much random environment
e.g. $G(n,n)$ and $G(N,N)$, $n<N$, share. It can also be seen from the slow de-correlation phenomenon, see \cite{Fer}, \cite{CorFerPec}. Looking at (\ref{hrescaled}) we see that
we have the fluctuation exponent $1/3$ (fluctuations have order $T^{1/3}$), the spatial correlation exponent $2/3$, and we also have the time correlation exponent $1=3/3$
as explained. This is the KPZ 1:2:3 scaling. For further references and more on random growth models in the KPZ-universality class and related interacting particle systems, we refer to the survey papers \cite{BorPet}, \cite{Corw} and \cite{Quas}.

The main result of the present paper is a limit theorem for the following two-time probability.
Fix $m,M,n,N$ with $1\le m<M$ and $1\le n<N$. For $a,A\in\mathbb{Z}$, we will consider the probability
\begin{equation}\label{paA}
P(a,A)=\mathbb{P}[G(m,n)< a,\,G(M,N)<A],
\end{equation}
in the appropriate scaling limit. The result is formulated in Theorem \ref{ThMain} below.

The first studies of the two-time problem, using a non-rigorous based on the replica method, was given by Dotsenko in \cite{Dots1}, \cite{Dots2}, see also \cite{Dots3}. However, the formulas are believed not to be correct, \cite{NarDou}. The replica method has also been
used by De Nardis and Le Doussal, \cite{NarDou}, to derive very interesting results in the limit $t_1/t_2\to 1$ and, for arbitrary $t_1/t_2$, in the partial tail of the joint law of 
$\mathcal{H}_T(\eta_1,t_1)$ and $\mathcal{H}_T(\eta_2,t_2)$ when $\mathcal{H}_T(\eta_1,t_1)$ is large positive.
In \cite{Dous}, Le Doussal gives a conjecturally exact formula for the limit $t_1/t_2\to 0$. See also \cite{FerSpo} for some rigorous work on this with quantitative results for the height
correlation in the stationary case, which is not investigated here.
We will not discuss these limits although to do so would be interesting. There are very interesting experimental and numerical results on the two-time problem by K. A. Takeuchi and collaborators, see \cite{Take}, \cite{TaSa} and \cite{NaDoTa}.

Recently there has been a striking new development on the two-time problem, and more generally the multi-time problem, by J. Baik and Z. Liu, \cite{BaiLiu}. They consider the totally asymmetric simple exclusion process (TASEP) in a circular geometry, the periodic TASEP. Baik and Liu are able to give formulas for the multi-time distribution as contour integrals of Fredholm determinants, and take the scaling limit in the so-called relaxation time scale, $T=O(L^{3/2})$, where $L$ is the period. In principle their formulas include the problem studied here, but they are not able to take the scaling limit that we study in this paper. It would be interesting to understand the relation between the two approaches.
For some comments on the multi-time problem in the setting used here see Remark \ref{remmultitime}. A related problem is to understand the Markovian time evolution of the whole limiting process with some fixed initial condition, the so called KPZ-fixed point. There has recently been very interesting progress on this problem by Matetski, Quastel and Remenik,
see \cite{MaQuRe} and \cite{MaQu}.

An outline of the paper is as follows. In section \ref{secresults} we give the formula for the two-time distribution using an integral of a Fredholm determinant and state the main theorem. The main theorem is proved in section \ref{secproofmain} using a sequence of lemmas proved in sections \ref{secprooflemmas} and \ref{secasymptotics}.
In section \ref{secoldformula}, we briefly discuss the relation to the result in our previous work \cite{JoTt}.

\bigskip
\noindent
{\bf Notation} Throughout the paper $1(\cdot)$ denotes an indicator function, $\gamma_r(a)$  is a positively oriented circle of radius $r$ around the point $a$, and $\gamma_r=\gamma_r(0)$. Also, $\Gamma_c$ is the upward oriented straight line through the point $c$, $t\mapsto c+it$, $t\in\mathbb{R}$.

\section{Results}\label{secresults}

Let $0<t_1<t_2$, $\eta_1,\eta_2\in\mathbb{R}$ and $\xi_1,\xi_2\in\mathbb{R}$ be given.
Furthermore $T$ is a parameter that will tend to infinity. To formulate the scaling limit we need the constants,
\begin{equation}\label{scalingconstants}
c_0=q^{-1/3}(1+\sqrt{q})^{1/3},\quad c_1=q^{-1/6}(1+\sqrt{q})^{2/3},\quad
c_2=\frac{2\sqrt{q}}{1-\sqrt{q}},\quad c_3=\frac{q^{1/6}(1+\sqrt{q})^{1/3}}{1-\sqrt{q}}
\end{equation}
We will investigate the asymptotics of the probability distribution defined by (\ref{paA}).
The appropriate scaling is then
\begin{align}\label{scaling}
n&=t_1T-c_1\eta_1(t_1T)^{2/3},\quad m=t_1T+c_1\eta_1(t_1T)^{2/3}\\
N&=t_2T-c_1\eta_2(t_2T)^{2/3},\quad M=t_2T+c_1\eta_2(t_2T)^{2/3}\notag\\
a&=c_2t_1T+c_3\xi_1(t_1T)^{1/3}, \quad A=c_2t_2T+c_3\xi_2(t_2T)^{1/3}.\notag
\end{align}
Let $\Delta t=t_2-t_1$, and write
\begin{equation}\label{alpha}
\alpha=\left(\frac{t_1}{\Delta t}\right)^{1/3}.
\end{equation}
Introduce the notation
\begin{equation}\label{deltaeta}
\Delta\eta=\eta_2\left(\frac{t_2}{\Delta t}\right)^{2/3}-\eta_1\left(\frac{t_1}{\Delta t}\right)^{2/3},\quad \Delta\xi=\xi_2\left(\frac{t_2}{\Delta t}\right)^{1/3}-\xi_1\left(\frac{t_1}{\Delta t}\right)^{1/3}.
\end{equation}

We will now define the limiting probability function. Before we can do that we need to define some functions. Fix $\delta$ such that
\begin{equation}\label{deltacondition}
\delta>\max(\eta_1,\alpha\Delta\eta),
\end{equation}
and define
\begin{align}\label{S1}
S_1(x,y)&=-\alpha e^{(\eta_1-\delta)x+(\delta-\alpha\Delta\eta)y}\int_0^\infty e^{(\alpha\Delta\eta-\eta_1)s}
K_{\text{Ai}}(\xi_1+\eta_1^2-s,\xi_1+\eta_1^2-x)\notag\\
&\times K_{\text{Ai}}(\Delta\xi+\Delta\eta^2+\alpha s,\Delta\xi+\Delta\eta^2+\alpha y)\,ds,
\end{align}
\begin{align}\label{T1}
T_1(x,y)&=\alpha e^{(\eta_1-\delta)x+(\delta-\alpha\Delta\eta)y}
\int_{-\infty}^0 e^{(\alpha\Delta\eta-\eta_1)s}
K_{\text{Ai}}(\xi_1+\eta_1^2-s,\xi_1+\eta_1^2-x)\notag
\\&\times K_{\text{Ai}}(\Delta\xi+\Delta\eta^2+\alpha s,\Delta\xi+\Delta\eta^2+\alpha y)\,ds,
\end{align}
\begin{equation}\label{S2}
S_2(x,y)=\alpha e^{(\delta-\alpha\Delta\eta)(y-x)}K_{\text{Ai}}(\Delta\xi+\Delta\eta^2+\alpha x,\Delta\xi+\Delta\eta^2+\alpha y),
\end{equation}
and
\begin{equation}\label{S3}
S_3(x,y)=e^{(\delta-\eta_1)(y-x)}K_{\text{Ai}}(\xi_1+\eta_1^2-x,\xi_1+\eta_1^2-y).
\end{equation}
Using these, we can define the functions
\begin{equation}\label{Sxy}
S(x,y)=S_1(x,y)+1(x> 0)S_2(x,y)-S_3(x,y)1(y<0),
\end{equation}
\begin{equation}\label{Txy}
T(x,y)=-T_1(x,y)-1(x>0)S_2(x,y)+S_3(x,y)1(y< 0).
\end{equation}

Let $u$ be a complex parameter and set
\begin{equation}\label{Ru}
R(u)(x,y)=S(x,y)+u^{-1}T(x,y).
\end{equation}
Consider the space
\begin{equation}\label{Xspace}
X=L^2(\mathbb{R}_-,dx)\oplus L^2(\mathbb{R}_+,dx),
\end{equation}
and define the following matrix kernel on $X$,
\begin{equation}\label{Kuv}
K(u)(x,y)=\begin{pmatrix} R_u(x,y) & R_u(x,y)  \\
                                              uR_u(x,y) & uR_u(x,y)
                                              
                    \end{pmatrix}.
 \end{equation}
 $K(u)$ defines a trace-class operator on $X$, which we also denote by $K(u)$. Let $\gamma_r$ denote a circle around the origin of radius $r$ with positive orientation.
 We define the two-time probability distribution by
 \begin{equation}\label{ftt}
F_{\text{two-time}}(\xi_1,\eta_1;\xi_2,\eta_2;\alpha)=\frac{1}{\pii}\int_{\gamma_r}\frac 1{u-1}\det(I+K(u))_X\,du,
\end{equation}
where $r>1$.

We can now formulate our main theorem.

\begin{theorem}\label{ThMain}
Let $P(a,A)$ be defined as in (\ref{paA}) and consider the scaling (\ref{scaling}). Then,
\begin{equation}\label{limit}
\lim_{T\to\infty} P(a,A)=F_{\text{two-time}}(\xi_1,\eta_1;\xi_2,\eta_2;\alpha).
\end{equation}
\end{theorem}

The theorem will be proved in section \ref{secproofmain}. The fact that $K(u)$ is a trace-class operator is Lemma \ref{lemtraceclass} below.

The formula for the two-time distribution can be written in different ways. In section \ref{SecFormulas}, we will give formulas suitable for studying the limits
$\alpha\to 0$, $\alpha\to \infty$ and expansions in $\alpha$ and $1/\alpha$ respectively. We will not discuss these expansions here, but refer to \cite{DJL} for
more on this and comparison with the results in \cite{Dous}.

For comments on the relation between this formula and the formula derived in \cite{JoTt}, see the discussion in section \ref{secoldformula}. 

\begin{remark}\label{remmultitime} 
{\rm It would be interesting to be able to prove the same type of scaling limit for the multi-time case, i.e. to consider the probability function
\begin{equation*}
P(a_1,\dots,a_{L})=\mathbb{P}\left[G(m_1,n_1)<a_1,\dots,G(m_{L},n_{L})<a_{L}\right],
\end{equation*}
where  $m_1< m_2<\dots< m_L$, and $n_1< n_2<\dots< n_L$. It is possible to write a formula analogous to (\ref{paAformula4}) below but with $L-1$ contour integrals.  This can be proved in a very similar way as the proof of (\ref{paAformula4}). We hope to say more on this problem in future work.}
\end{remark}

\section{Proof of the Main theorem}\label{secproofmain}
In this section we will prove the main theorem. Along the way we will use several lemmas that will be proved in sections \ref{secprooflemmas} and \ref{secasymptotics}.

Write
\begin{equation}\label{Gm}
\mathbf{G}(m)=(G(m,1),\dots,G(m,N)),
\end{equation}
for $m\ge 0$, and a fixed $N\ge 1$. Let $\mathbf{G}(0)=0$. By $\Delta$ we denote the finite difference operator defined on functions $f:\mathbb{Z}\mapsto\mathbb{C}$ by
$\Delta f(x)=f(x+1)-f(x)$, which has the inverse
$$
\Delta^{-1} f(x)=\sum_{y=-\infty}^{x-1} f(y),
$$
for all functions $f$ for which the series converges. The negative binomial weight is
\begin{equation}\label{negbinom}
w_m(x)=(1-q)^m\binom{x+m-1}{x} q^x 1(x\ge 0),
\end{equation}
for $m\ge 1$, $x\in\mathbb{Z}$. Write
\begin{equation}\label{WN}
W_N=\{\mathbf{x}=(x_1,\dots,x_N)\in\mathbb{Z}^N\,,\,x_1\le\dots\le x_N\}.
\end{equation}
Note that $\mathbf{G}(m)\in W_N$.

The following proposition is the starting point for the proof. It is proved in \cite{JoMar} following the paper \cite{Warr} by J. Warren, see also \cite{DiWa} for a more systematic treatment.
\begin{proposition}\label{PropTransition}
The vectors $(\mathbf{G}(m))_{m\ge 0}$ form a Markov chain with transition function
\begin{equation}\label{Trans}
\mathbb{P}[\mathbf{G}(m)=\mathbf{y}\,|\,\mathbf{G}(\ell)=\mathbf{x}]=\det(\Delta^{j-i}w_{m-\ell}(y_j-x_i))_{1\le i,j\le N},
\end{equation}
for any $\mathbf{x},\mathbf{y}\in W_N$, $m>\ell\ge 0$.
\end{proposition}

Write
\begin{equation}\label{deltam}
\Delta m=M-m,\quad\Delta N=N-n,\quad \Delta a=A-a,
\end{equation}
and
\begin{equation*}
W_{N,n}(a)=\{\mathbf{x}\in W_N;\,x_n< a\}.
\end{equation*}
We can the write
\begin{equation}\label{paAformula}
P(a,A)=\sum_{\mathbf{x}\in W_{N,n}(a)}\sum_{\mathbf{y}\in W_{N,N}(A)}\det(\Delta^{j-i}w_m(x_j))_{1\le i,j\le N}\det(\Delta^{j-i}w_{\Delta m}(y_j-x_i))_{1\le i,j\le N}.
\end{equation}
Here we would like to perform the sum over $\mathbf{y}$, which is straightforward, and then the sum over $\mathbf{x}$, which is tricky since we cannot use the Cauchy-Binet
identity directly. An important step is part a) of the following lemma, which is proved in section \ref{secprooflemmas}. The proof of (\ref{Partsum}) uses successive summations by parts and
generalizes the proof of Lemma 3.2 in \cite{JoMar}.

\begin{lemma}\label{LemSumparts} Let $f,g:\mathbb{Z}\mapsto\mathbb{R}$ be given functions and assume that there is an $L\in\mathbb{Z}$ such that $f(x)=g(x)=0$ if $x<L$.

(a) Let $a_i,d_i\in\mathbb{Z}$, $1\le i\le N$ and fix $k$, $1\le k\le N$. Then,
\begin{align}\label{Partsum}
&\sum_{\mathbf{x}\in W_{N,k}(a)}\det\big(\Delta^{j-a_i}f(x_j-y_i)\big)_{1\le i,j\le N}\det\big(\Delta^{d_i-j}g(z_i-x_j)\big)_{1\le i,j\le N}\\
=&\sum_{\mathbf{x}\in W_{N,k}(a)}\det\big(\Delta^{k-a_i}f(x_j-y_i)\big)_{1\le i,j\le N}\det\big(\Delta^{d_i-k}g(z_i-x_j)\big)_{1\le i,j\le N}.\notag
\end{align}

(b) For $1\le n\le N$, we have the identity
\begin{equation}\label{secondsum}
\sum_{\mathbf{x}\in W_{N,N}(A)}\det\big(\Delta^{i-n}w_m(x_i-y_j)\big)_{1\le i,j\le N}=\det\big(\Delta^{i-n-1}w_m(A-y_j)\big)_{1\le i,j\le N}.
\end{equation}
\end{lemma}
If we use (\ref{Partsum}) and (\ref{secondsum}) in (\ref{paAformula}), we find
\begin{equation}\label{paAformula2}
P(a,A)=\sum_{\mathbf{x}\in W_{N,n}(a)}\det\big(\Delta^{n-i}w_m(x_j)\big)_{1\le i,j\le N}\det\big(\Delta^{j-n-1}w_{\Delta m}(A-x_i)\big)_{1\le i,j\le N}.
\end{equation}

Before we show how we can use the Cauchy-Binet identity to do the summation in (\ref{paAformula2}), we will modify it somewhat. Below, this modification will be a kind of
orthogonalization procedure, and will be important for obtaining a Fredholm determinant. Let $A=(a_{ij})$ and $B=(b_{ij})$ be two $N\times N$-matrices that satisfy
$a_{ij}=0$ if $j>i$ and $b_{ij}=0$ if $j<i$, so that $A$ is lower- and $B$ upper-triangular. Assume that
\begin{equation}\label{detAB}
\det AB=\prod_{i=1}^Na_{ii}b_{ii}=1.
\end{equation}
For $x\in\mathbb{Z}$, $1\le i,j\le N$, we define
\begin{equation}\label{f01}
f_{0,1}(i,x)=\sum_{k=1}^Na_{ik}(-1)^n\Delta^{n-k}w_m(x+a),
\end{equation}
and
\begin{equation}\label{f12}
f_{1,2}(x,j)=\sum_{k=1}^N(-1)^n\Delta^{k-1-n}w_{\Delta m}(\Delta a-x)b_{kj},
\end{equation}
where $w_m$ is the negative binomial weight (\ref{negbinom}).
If we shift $x_i\to x_i+a$, $1\le i\le N$, in (\ref{paAformula2}), and use (\ref{detAB}), (\ref{f01}) and (\ref{f12}), we get
\begin{equation}\label{paAformula3}
P(a,A)=\sum_{\mathbf{x}\in W_{N,n}(0)}\det\big(f_{0,1}(i,x_j)\big)_{1\le i,j\le N}\det\big(f_{1,2}(x_i,j)\big)_{1\le i,j\le N}.
\end{equation}
This formula is the basis for the next lemma, the proof of which is based on the Cauchy-Binet identity. However, because of the restriction $x_n< 0$ in the summation in
(\ref{paAformula3}), we cannot apply the identity directly. In order to state the result we need some further notation. Define
\begin{equation}\label{L1}
L_1(i,j)=\sum_{x=-\infty}^{-1}f_{0,1}(i,x)f_{1,2}(x,j),
\end{equation}
\begin{equation}\label{L2}
L_2(i,j)=\sum_{x=0}^{\infty}f_{0,1}(i,x)f_{1,2}(x,j).
\end{equation}
Let $u$ be a complex parameter and set
\begin{equation}\label{Lijuv}
L(i,j;u)=u^{1(i>n)}L_1(i,j)+u^{-1(i\le n)}L_2(i,j).
\end{equation}

\begin{lemma}\label{LemCB}
We have the formula,
\begin{equation}\label{paAformula4}
P(a,A)=\frac 1{\pii}\int_{\gamma_r}\frac 1{u-1}\det\big(L(i,j;u)\big)_{1\le i,j\le N}du,
\end{equation}
for any $r>1$.
\end{lemma}

The lemma is proved in section \ref{secprooflemmas}. The contour integral come from the need to capture the restriction $x_n<0$ and still use the Cauchy-Binet
identity.

We now come to the choice of the matrices $A$ and $B$. The aim is to get a good formula for $f_{0,1}$ and $f_{1,2}$ and make it possible to write the determinant in 
(\ref{paAformula4}) as a Fredholm determinant suitable for asymptotic analysis. Define
\begin{equation}\label{Hnmx}
H_{n,m,x}(w)=\frac{w^n(1-w)^{x+m}}{\big(1-\frac{w}{1-q}\big)^m}.
\end{equation}
Using a generating function for the negative binomial weight (\ref{negbinom}), it is straightforward to show that for all $m\ge 1$, $k,x\in\mathbb{Z}$,
\begin{equation}\label{Deltakwm}
\Delta^n w_m(x)=\frac{(-1)^{k-1}}{\pii}\int_{\gamma_r}H_{n,m,x}(z)\frac{dz}{1-z},
\end{equation}
if $r>1$. For $k,x\in\mathbb{Z}$, $m\ge 1$, $\epsilon\in\{0,1\}$ and $0<\tau<1$, we define
\begin{equation}\label{beta}
\beta_k^{\epsilon}(m,a)=\frac 1{\pii}\int_{\gamma_\tau}\zeta^{k-1}\frac{\left(1-\frac{\zeta}{1-q}\right)^m}{(1-\zeta)^{a+m-\epsilon}}d\zeta.
\end{equation}
Note that $\beta_0^{\epsilon}=1$ and $\beta_k^{\epsilon}=0$ if $k\ge 1$. By expanding $(z-\zeta)^{-1}$ in powers of $\zeta/z$, we see that
\begin{equation}\label{betaformula}
\sum_{k=1}^N\frac{\beta_{k-i}^{\epsilon}(m,a)}{z^k}=\frac 1{\pii}\int_{\gamma_\tau}\frac{(1-\zeta)^\epsilon}{H_{i,m,a}(\zeta)(z-\zeta)}\,d\zeta,
\end{equation}
provided $|z|>\tau$. 

We now define the matrices $A$ and $B$. Let $c(i)$ be a conjugation factor defined below in (\ref{conjfactor}) which we need to make the asymptotic analysis work. Set
\begin{equation}\label{ABdef}
a_{ik}=c(i)(-1)^{-k}\beta_{k-i}^1(m,a),\quad b_{kj}=c(j)^{-1}(-1)^k\beta_{j-k}^0(\Delta m,\Delta a).
\end{equation}
From the properties of $\beta_k^\epsilon$, we see that $(a_{ik})$ is lower- and $(b_{kj})$ upper-triangular, and that the condition (\ref{detAB}) is satisfied.

\begin{lemma}\label{Lemf01f12}
If $f_{0,1}$ and $f_{1,2}$ are defined by (\ref{f01}) and (\ref{f12}) respectively, and $a_{ik}$ and $b_{kj}$ by (\ref{ABdef}), then
\begin{equation}\label{f01formula}
f_{0,1}(i,x)=-\frac {c(i)}{(\pii)^2}\int_{\gamma_r}dz\int_{\gamma_\tau}d\zeta\frac{H_{n,m,a+x}(z)(1-\zeta)}{H_{i,m,a}(\zeta)(z-\zeta)(1-z)},
\end{equation}
\begin{equation}\label{f12formula}
f_{1,2}(x,j)=\frac {c(j)^{-1}}{(\pii)^2}\int_{\gamma_r}dw\int_{\gamma_\tau}d\omega\frac{H_{\Delta n,\Delta m,\Delta a-x}(w)}{H_{N+1-j,\Delta m,\Delta a}(\omega)(w-\omega)(1-w)},
\end{equation}
where $0<\tau<1<r$.
\end{lemma}

The proof of the lemma, which will be given in section \ref{secprooflemmas}, is a straightforward computation using the definitions and (\ref{betaformula}).

We now turn to rewriting the determinant in (\ref{paAformula4}) as a Fredholm determinant and performing the asymptotic analysis. The conjugation factor $c(i)$ in (\ref{ABdef})
is given by
\begin{equation}\label{conjfactor}
c(i)=(1-\sqrt{q})^i e^{-\delta i/c_0(t_1T)^{1/3}},
\end{equation}
where $\delta>0$ is fixed, and satisfies (\ref{deltacondition}), and $c_1$ is given by (\ref{scalingconstants}). Let $\tau_1, \tau_2,\rho_1,\rho_2$ and $\rho_3$ be radii such that
\begin{equation}\label{radii}
0<\tau_1,\tau_2<1-\rho_1<1-\rho_2<1-\rho_3<1-q.
\end{equation}
We denote by $\gamma_\rho(1)$ a positively oriented circle around the point $1$ with radius $\rho$. For $\epsilon\in\{0,1\}$ and $1\le i,j\le N$, we define
\begin{equation}\label{A1}
A_1(i,j)=\frac{c(i)}{c(j)(\pii)^4}\int_{\gamma_{\rho_1}(1)}dz\int_{\gamma_{\rho_2}(1)}dw\int_{\gamma_{\tau_1}}d\zeta\int_{\gamma_{\tau_2}}d\omega
\frac{H_{n,m,a}(z)H_{\Delta n,\Delta m,\Delta a}(w)(1-\zeta)(1-z)^{-1}}{H_{i,m,a}(\zeta)H_{N+1-j,\Delta m,\Delta a}(\omega)(z-\zeta)(w-\omega)(z-w)},
\end{equation}
\begin{equation}\label{B1}
B_1(i,j)=\frac{c(i)}{c(j)(\pii)^4}\int_{\gamma_{\rho_3}(1)}dz\int_{\gamma_{\rho_2}(1)}dw\int_{\gamma_{\tau_1}}d\zeta\int_{\gamma_{\tau_2}}d\omega
\frac{H_{n,m,a}(z)H_{\Delta n,\Delta m,\Delta a}(w)(1-\zeta)^(1-z)^{-1}}{H_{i,m,a}(\zeta)H_{N+1-j,\Delta m,\Delta a}(\omega)(z-\zeta)(w-\omega)(z-w)},
\end{equation}
\begin{equation}\label{A2}
A_2(i,j)=\frac{c(i)}{c(j)(\pii)^2}\int_{\gamma_{\rho_2}(1)}dw\int_{\gamma_{\tau_2}}d\omega
\frac{H_{N-i,\Delta m,\Delta a}(w)}{H_{N+1-j,\Delta m,\Delta a}(\omega)(w-\omega)},
\end{equation}
and
\begin{equation}\label{A3}
A_3(i,j)=\frac{c(i)}{c(j)(\pii)^2}\int_{\gamma_{\rho_1}(1)}dz\int_{\gamma_{\tau_1}}d\zeta
\frac{H_{j-1,m,a}(z)(1-\zeta)}{H_{i,m,a}(\zeta)(z-\zeta)(1-z)}.
\end{equation}
We also define, for $\epsilon\in\{0,1\}$ and $1\le i,j\le N$,
\begin{equation}\label{Cepsilon}
C(i,j)=A_1(i,j)-1(i>n)A_2(i,j)+A_3(i,j)1(j\le n),
\end{equation}
\begin{equation}\label{Depsilon}
D(i,j)=-B_1(i,j)+1(i>n)A_2(i,j)-A_3(i,j)1(j\le n),
\end{equation}
compare with (\ref{Sxy}) and (\ref{Txy}). 

We can now express $L_p$, $p=1,2$, in terms of these objects.
\begin{lemma}\label{LemLformulas} We have the formulas
\begin{equation}\label{L1delta}
L_1(i,j)=1(i\le n)\delta_{ij}+C(i,j),
\end{equation}
and
\begin{equation}\label{L2delta}
L_2(i,j)=1(i>n)\delta_{ij}+D(i,j).
\end{equation}
\end{lemma}
The proof is based on (\ref{L1}), (\ref{L2}), and Lemma \ref{Lemf01f12}, and suitable contour deformations in order to get the contours into
positions that can be used in the asymptotic analysis, see section \ref{secprooflemmas}.

Combining (\ref{Lijuv}) with Lemma \ref{LemLformulas} we obtain
\begin{equation}\label{Liju}
L(i,j;u)=\delta_{ij}+M_u(i,j),
\end{equation}
where
\begin{equation}\label{Muij}
M_u(i,j)=u^{-1(i\le n)}\left(uC(i,j)+D(i,j)\right),
\end{equation}
and we also set $M_{u}(i,j)=0$ if $i,j\notin\{1,\dots,N\}$.
Thus we have the formula
\begin{equation}\label{paAformula5}
P(a,A)=\frac 1{\pii}\int_{\gamma_r}\frac 1{u-1}\det\big(\delta_{ij}+M_u(i,j)\big)_{1\le i,j\le N}du.
\end{equation}

Next, we want to rewrite the determinant in (\ref{paAformula5}) in a block determinant form, corresponding to $i\le n$ and $i>n$, and similarly for $j$. For $r,s\in\{1,2\}$, and $x,y\in\mathbb{R}$, we define
\begin{equation}\label{Fuv}
F_{u}(r,x;s,y)=M_{u}(n+[x]+1,n+[y]+1),
\end{equation}
where $[\cdot]$ denotes the integer part. The right side of (\ref{Fuv}) does not depend on $r$ or $s$ explicitely but we have $x<0$ for $r=1$ and $x\ge 0$ for $r=2$, and
correspondingly for $y$ depending on $s$.
Let $\Lambda=\{1,2\}\times\mathbb{R}$ and define the measures
\begin{equation*}
d\nu_1(x)=1(x<0)dx,\quad d\nu_2(x)=1(x\ge 0)(x)dx.
\end{equation*}
On $\Lambda$ we define a measure $\rho$ by
\begin{equation}\label{rhomeasure}
\int_{\Lambda}f(\lambda)d\rho(\lambda)=\sum_{r=1}^2\int_{\mathbb{R}}f(r,x)\,d\nu_r(x),
\end{equation}
for every integrable function $f:\Lambda\mapsto\mathbb{R}$.
$F_{u}$ defines an integral operator $F_{u}$ on $L^2(\Lambda,\rho)$ with kernel $F_{u}(r,x;s,y)$. Note that the space  $L^2(\Lambda,\rho)$ is isomorphic to the space 
$X$ defined in (\ref{Xspace}), and we can also think of $F_{u}$ as a matrix operator.

\begin{lemma}\label{LemBlockdet}
We have the identity,
\begin{equation}\label{detidentity}
\det(\delta_{ij}+M_{u}(i,j))_{1\le i,j\le N}=\det(I+F_{u})_{L^2(\Lambda,\rho)}.
\end{equation}
\end{lemma}

This is straightforward, using Fredholm expansions,  and the lemma will be proved in section \ref{secprooflemmas}.

We can now insert the formula (\ref{detidentity}) into (\ref{paAformula5}). This leads to a formula that can be used for taking a limit, but before considering the limit, we have to
introduce the appropriate scalings. For $s=1,2$, we define
\begin{equation}\label{Fuvtilde}
\tilde{F}_{u,T}(r,x;s,y)=
c_0(t_1T)^{1/3}F_{u}(r,c_0(t_1T)^{1/3}x;s,c_0(t_1T)^{1/3}y) 
\end{equation}
where $c_0$ is given by (\ref{scalingconstants}).
The next lemma follows from (\ref{paAformula5}), Lemma \ref{LemBlockdet}, and (\ref{Fuvtilde}), see section \ref{secprooflemmas}.
\begin{lemma}\label{LemRescaled}
We have the formula,
\begin{equation}\label{paAformula6}
P(a,A)=\frac 1{\pii}\int_{\gamma_r}\frac 1{u-1}\det\big(I+\tilde{F}_{u,T}\big)_{L^2(\Lambda,\rho)}du.
\end{equation}
\end{lemma}
Theorem \ref{ThMain} now follows by combining this lemma with the next lemma which will be proved in section \ref{secasymptotics}.
\begin{lemma}\label{LemScalinglimit}
Consider the scaling (\ref{scaling}) and let $K(u)$ be the matrix kernel defined by (\ref{Kuv}). Then,
\begin{equation}\label{deterlimit}
\lim_{T\to\infty}\det\big(I+\tilde{F}_{u,T}\big)_{L^2(\Lambda,\rho)}=
\det\big(I+K(u)\big)_X,
\end{equation}
uniformly for $u$ in a compact set.
\end{lemma}

\section{Proof of Lemmas}\label{secprooflemmas}
In this section we will prove the lemmas that were used in section \ref{secproofmain}. Some results related to the asymptotic analysis will be proved in section \ref{secasymptotics}.
\begin{proof}[Proof of Lemma \ref{LemSumparts}]
Write
\begin{equation*}
W^*_{N,k}(a)=\{\mathbf{x}\in W_N\,;\,x_k=a\}
\end{equation*}
so that
\begin{equation*}
W_{N,k}(a)=\bigcup_{t=-\infty}^a W^*_{N,k}(t)
\end{equation*}
Hence, it is enough to prove the statement with $W_{N,k}(a)$ replaced by $W^*_{N,k}(t)$.
Let $a_i, b_i, c_i, d_i\in\mathbb{Z}$, $1\le i,j\le N$, and let $k<\ell\le N$. Assume that $b_{\ell-1}=b_\ell-1$, and $c_{\ell}=c_{\ell+1}$ if $\ell<N$. Set
\begin{equation*}
b_j'=\begin{cases}b_j &\text{if $j\neq \ell$} \\ b_\ell-1 &\text{if $j=\ell$} \end{cases}, \quad
c_j'=\begin{cases}c_j &\text{if $j\neq \ell$} \\ c_\ell-1 &\text{if $j=\ell$} \end{cases}.
\end{equation*}
Then,
\begin{align}\label{sum1}
&\sum_{\mathbf{x}\in W^*_{N,k}(t)}\det\left(\Delta^{b_j-a_i}f(x_j-y_i)\right)_{1\le i,j\le N}\det\left(\Delta^{d_i-c_j}g(z_i-x_j)\right)_{1\le i,j\le N}\\
=&\sum_{\mathbf{x}\in W^*_{N,k}(t)}\det\left(\Delta^{b'_j-a_i}f(x_j-y_i)\right)_{1\le i,j\le N}\det\left(\Delta^{d_i-c'_j}g(z_i-x_j)\right)_{1\le i,j\le N}.\notag
\end{align}

To prove (\ref{sum1}), we use the summation by parts identity,
\begin{equation}\label{sumparts}
\sum_{y=a}^b\Delta u(y-x)c(z-y)=\sum_{y=a}^b u(y-x)\Delta c(z-y) + u(b+1-x)v(z-b)-u(a-x)v(z+1-a).
\end{equation}
Consider the $x_\ell$-summation in the left side of (\ref{sum1}) with all the other variables fixed. Let $x_{\ell+1}=\infty$ if $\ell=N$ and let $\Delta_x$ denote the finite difference with respect to the variable $x$. Using (\ref{sumparts}) in the second inequality we get
\begin{align}\label{sumpartsdet}
&\sum_{x_\ell=x_{\ell-1}}^{x_{\ell+1}}\det\left(\Delta^{b_j-a_i}f(x_j-y_i)\right)_{1\le i,j\le N}\det\left(\Delta^{d_i-c_j}g(z_i-x_j)\right)_{1\le i,j\le N}\\
=&\sum_{x_\ell=x_{\ell-1}}^{x_{\ell+1}}\Delta_{x_{\ell}}\det\left(\Delta^{b'_j-a_i}f(x_j-y_i)\right)_{1\le i,j\le N}\det\left(\Delta^{d_i-c_j}g(z_i-x_j)\right)_{1\le i,j\le N}\notag\\
=&\sum_{x_\ell=x_{\ell-1}}^{x_{\ell+1}}\det\left(\Delta^{b'_j-a_i}f(x_j-y_i)\right)_{1\le i,j\le N}\det\left(\Delta^{d_i-c'_j}g(z_i-x_j)\right)_{1\le i,j\le N}\notag\\
+&\left.\det\left(\Delta^{b_j'-a_i}f(x_j-y_i)\right)_{1\le i,j\le N}\right|_{x_\ell\to x_{\ell+1}+1}\left.\det\left(\Delta^{d_i-c_j}g(z_i-x_j)\right)_{1\le i,j\le N}\right|_{x_\ell\to x_{\ell+1}}\notag\\
-&\left.\det\left(\Delta^{b_j'-a_i}f(x_j-y_i)\right)_{1\le i,j\le N}\right|_{x_\ell\to x_{\ell-1}}\left.\det\left(\Delta^{d_i-c_j}g(z_i-x_j)\right)_{1\le i,j\le N}\right|_{x_\ell\to x_{\ell-1}-1}.\notag
\end{align}
If $\ell=N$, then the first boundary term in (\ref{sumpartsdet}) is $=0$. This follows since $\Delta^{d_i-c_\ell}g(z_i-\infty)=0$ (assumption that all series are convergents, expressions well-defined), so one column in the second determinant the first boundary term in (\ref{sumpartsdet}) is $=0$. If $\ell<N$, then the first boundary term in (\ref{sumpartsdet}) is $=0$
because $c_{\ell}=c_{\ell+1}$, and $x_\ell\to x_{\ell+1}$ means that columns $\ell$ and $\ell+1$ will be identical in the second determinant. Since $b'_{\ell}=b_\ell-1=b_{\ell-1}$, we see that columns $\ell$ and
$\ell-1$ in the first determinant in the second boundary term in (\ref{sumpartsdet}) will be identical.

Similarly, if $1\le \ell<k$, and $c_{\ell+1}=c_\ell+1$, $b_\ell=b_{\ell-1}$, then
\begin{align}\label{sum2}
&\sum_{\mathbf{x}\in W^*_{N,k}(t)}\det\left(\Delta^{b_j-a_i}f(x_j-y_i)\right)_{1\le i,j\le N}\det\left(\Delta^{d_i-c_j}g(z_i-x_j)\right)_{1\le i,j\le N}\\
=&\sum_{\mathbf{x}\in W^*_{N,k}(t)}\det\left(\Delta^{b''_j-a_i}f(x_j-y_i)\right)_{1\le i,j\le N}\det\left(\Delta^{d_i-c''_j}g(z_i-x_j)\right)_{1\le i,j\le N},\notag
\end{align}
where
\begin{equation*}
b_j''=\begin{cases}b_j &\text{if $j\neq \ell$} \\ b_\ell+1 &\text{if $j=\ell$} \end{cases}, \quad
c_j''=\begin{cases}c_j &\text{if $j\neq \ell$} \\ c_\ell+1 &\text{if $j=\ell$} \end{cases}.
\end{equation*}
The proof of (\ref{sum2}) is analogous to the proof of (\ref{sum1}).

To prove lemma \ref{LemSumparts}, we apply (\ref{sum1}) successively to $x_N,x_{N-1},\dots,x_{k+1}$, and then to $x_N,x_{N-1},\dots,x_{k+2}$ etc.,
and then finally just to $x_N$. Similarly, we apply (\ref{sum2}) to $x_1,x_2,\dots, x_{k-1}$, then to $x_1,x_2,\dots, x_{k-2}$, and finally just to $x_1$.
This proofs part a) of the lemma.

Part b) of the lemma follows from the identity
\begin{equation}\label{sum3}
\sum_{\mathbf{x}\in W_{N,N}(a)}\det\left(\Delta^{i-n}f_j(x_i)\right)_{1\le i,j\le N}=\det\left(\Delta^{i-1-n}f_j(a+1)\right)_{1\le i,j\le N}.
\end{equation}
To prove (\ref{sum3}), first sum over $x_N$ from $x_{N-1}$ to $a$ in the last row. This gives $\Delta^{N-1-n}f_j(a+1)-\Delta^{N-1}f_j(x_{N-1})$. The last term does not contribute
since it is the same as in row $N-1$. We can now sum over $x_{N-1}$ from $x_{N-2}$ to $a$ in row $N-1$ etc. In this way we obtain (\ref{sum3}).
\end{proof}

\begin{proof}[Proof of Lemma \ref{LemCB}]
We see that
\begin{align}\label{QaAformula}
P(a,A)&=\sum_{\mathbf{x}\in W_N\,;\,x_n<0}\det\big(f_{0,1}(i,x_j)\big)_{1\le i,j\le N}\det\big(f_{1,2}(x_i,j)\big)_{1\le i,j\le N}\\
&=\sum_{\mathbf{x}\in W_N}\det\big(f_{0,1}(i,x_j)\big)_{1\le i,j\le N}\det\big(f_{1,2}(x_i,j)\big)_{1\le i,j\le N} 1\left(\sum_{j=1}^N1(x_j<0)\ge n\right)\notag
\end{align}
Now, for any $r>0$,
\begin{equation*}
\frac 1{\pii}\int_{\gamma_r}\frac{u^{\sum_{j=1}^N1(x_j<0)}}{u^{\ell+1}}du=1\left(\sum_{j=1}^N1(x_j<0)=\ell\right).
\end{equation*}
Summing over $\ell\ge n$ and assuming that $r>1$, we get
\begin{equation}\label{contourformula}
\frac 1{\pii}\int_{\gamma_r}\frac{u^{\sum_{j=1}^N1(x_j<0)}}{u^n(u-1)}du=1\left(\sum_{j=1}^N1(x_j<0)\ge n\right).
\end{equation}
Since,
\begin{equation*}
u^{\sum_{j=1}^N1(x_j<0)}=\prod_{j=1}^N\left(u1(x_j<0)+1(x_j\ge 0)\right),
\end{equation*}
it follows from (\ref{QaAformula}), (\ref{contourformula}), and the Cauchy-Binet identity that
\begin{align*}
P(a,A)&=\frac 1{\pii}\int_{\gamma_r}\frac{du}{u^n(u-1)}
\sum_{\mathbf{x}\in W_N}\det\big(f_{0,1}(i,x_j)\big)_{1\le i,j\le N}\det\big(f_{1,2}(x_i,j)\big)_{1\le i,j\le N}\\
&\times \prod_{j=1}^N\left(u1(x_j<0)+1(x_j\ge 0)\right)\notag\\
&=\frac 1{\pii}\int_{\gamma_r}\frac{du}{u^n(u-1)}\det\left(\sum_{z\in\mathbb{Z}}f_{0,1}(i,x)f_{1,2}(x,j)(u1(x<0)+1(x\ge 0))\right)_{1\le i,j\le N}\notag\\
&=\frac 1{\pii}\int_{\gamma_r}\frac{du}{u-1}\det\left(u^{-1(i\le n)}(uL_1(i,j)+L_2(i,j))\right)_{1\le i,j\le N}\notag\\
&=\frac 1{\pii}\int_{\gamma_r}\frac{du}{u-1}\det\left(L(i,j;u)\right)_{1\le i,j\le N}.\notag
\end{align*}

\end{proof}

\begin{proof}[Proof of Lemma \ref{Lemf01f12}]
It follows from (\ref{f01}), (\ref{betaformula}), and (\ref{ABdef}), that
\begin{align*}
f_{0,1}(i,x)&=c(i)\sum_{k=1}^N\beta_{k-i}^1(m,a)(-1)^{n-k}\Delta^{n-k}w_m(a+x)\\&=
-\frac {c(i)}{\pii}\int_{\gamma_r}\left(\sum_{k=1}^N\frac{\beta_{k-i}^1(m,a)}{z^k}\right)H_{n,m,a+x}(z)\frac{dz}{1-z}\\
&=-\frac {c(i)}{(\pii)^2}\int_{\gamma_r}dz\int_{\gamma_\tau}d\zeta\frac{H_{n,m,a+x}(z)(1-\zeta)}{H_{i,m,a}(\zeta)(z-\zeta)(1-z)}.
\end{align*}
Similarly, by (\ref{f12}), (\ref{betaformula}) and (\ref{ABdef}),
\begin{align*}
f_{1,2}(i,x)&=c(j)^{-1}\sum_{k=1}^N(-1)^{k-n}\Delta^{k-n-1}w_{\Delta m}(\Delta a-x)\beta_{j-k}^0(\Delta m,\Delta a)\\&=
\frac {c(j)^{-1}}{\pii}\int_{\gamma_r}\left(\sum_{k=1}^N\frac{\beta_{j-k}^0(\Delta m,\Delta a)}{w^{N+1-k}}\right)H_{\Delta n,\Delta m,\Delta a-x}(w)\frac{dw}{1-w}\\
&=\frac {c(j)^{-1}}{\pii}\int_{\gamma_r}\left(\sum_{k=1}^N\frac{\beta_{j-(N+1-k)}^0(\Delta m,\Delta a)}{w^{k}}\right)H_{\Delta n,\Delta m,\Delta a-x}(w)\frac{dw}{1-w}\\
&=\frac {c(j)^{-1}}{(\pii)^2}\int_{\gamma_r}dw\int_{\gamma_\tau}d\omega\frac{H_{\Delta n,\Delta m,\Delta a-x}(z)}{H_{N+1-j,\Delta m,\Delta a}(\omega)(w-\omega)(1-w)}.
\end{align*}
This proves the lemma.
\end{proof}

\begin{proof}[Proof of Lemma \ref{LemLformulas}]
Recall the condition (\ref{radii}) and choose $r_1, r_2$ so that $r_1>r_2>1+\max(\rho_1,\rho_2)$, which means that  $\gamma_{r_i}(1)$ surrounds $\gamma_{\rho_i}$ and $\gamma_{\tau_i}$, $i=1,2$.
It follows from (\ref{f01formula}) and (\ref{f12formula}), that
\begin{align*}
L_1(i,j)&=-\frac{c(i)c(j)^{-1}}{(\pii)^4}\sum_{x=-\infty}^{-1}\left(\int_{\gamma_{r_1}(1)}dz\int_{\gamma_{\tau_1}}d\zeta\frac{H_{n,m,a+x}(z)(1-\zeta)}{H_{i,m,a}(\zeta)(z-\zeta)(1-z)}\right)
\\
&\times\left(\int_{\gamma_{r_2}(1)}dw\int_{\gamma_{\tau_2}}d\omega\frac{H_{\Delta n,\Delta m,\Delta a-x}(w)}{H_{N+1-\Delta m,\Delta a}(\omega)(w-\omega)(1-w)}\right)\\
&=-\frac{c(i)c(j)^{-1}}{(\pii)^4}\int_{\gamma_{r_1}(1)}dz\int_{\gamma_{\tau_1}}d\zeta\int_{\gamma_{r_2}(1)}dw\int_{\gamma_{\tau_2}}d\omega\left(\sum_{x=-\infty}^{-1}
\left(\frac{1-z}{1-w}\right)^x\right)\\
&\times\frac{H_{n,m,a}(z)H_{\Delta n,\Delta m,\Delta a}(w)(1-\zeta)}{H_{i,m,a}(\zeta)H_{N+1-j,\Delta m,\Delta a}(\omega)(z-\zeta)(w-\omega)
(1-z)(1-w)}.
\end{align*}
Since $r_1>r_2$,
\begin{equation*}
\sum_{x=-\infty}^{-1}\left(\frac{1-z}{1-w}\right)^x=-\frac{1-w}{z-w},
\end{equation*}
and we obtain
\begin{align*}
L_1(i,j)&=\frac{c(i)c(j)^{-1}}{(\pii)^4}\int_{\gamma_{r_1}(1)}dz\int_{\gamma_{\tau_1}}d\zeta\int_{\gamma_{r_2}(1)}dw\int_{\gamma_{\tau_2}}d\omega
\\&\times \frac{H_{n,m,a}(z)H_{\Delta n,\Delta m,\Delta a}(w)(1-\zeta)}{H_{i,m,a}(\zeta)H_{N+1-j,\Delta m,\Delta a}(\omega)(z-\zeta)
(w-\omega)(z-w)(1-z)}.
\end{align*}
We now deform $\gamma_{r_2}(1)$ to $\gamma_{\rho_2}(1)$. Doing so, we cross the pole at $w=\omega$, and hence
\begin{align}\label{L1step}
L_1(i,j)&=\frac{c(i)c(j)^{-1}}{(\pii)^4}\int_{\gamma_{r_1}(1)}dz\int_{\gamma_{\tau_1}}d\zeta\int_{\gamma_{\rho_2}(1)}dw\int_{\gamma_{\tau_2}}d\omega
\\&\times \frac{H_{n,m,a}(z)H_{\Delta n,\Delta m,\Delta a}(w)(1-\zeta)}{H_{i,m,a}(\zeta)H_{N+1-j,\Delta m,\Delta a}(\omega)(z-\zeta)
(w-\omega)(z-w)(1-z)}\notag\\
&+\frac{c(i)c(j)^{-1}}{(\pii)^3}\int_{\gamma_{r_1}(1)}dz\int_{\gamma_{\tau_1}}d\zeta\int_{\gamma_{\tau_2}}d\omega
\frac{H_{n,m,a}(z)(1-\zeta)}{H_{i,m,a}(\zeta)\omega^{n+1-j}
(z-\omega)(z-\zeta)(1-z)}:=I_1+I_2.\notag
\end{align}
In $I_1$ we can shrink $\gamma_{r_1}(1)$ to $\gamma_{\rho_1}(1)$. We then cross the pole at $z=\zeta$ (but not $z=w$ since $\rho_2<\rho_1$). Thus,
by (\ref{A1}),
\begin{align}\label{I1}
I_1&=A_1(i,j)+\frac{c(i)c(j)^{-1}}{(\pii)^3}\int_{\gamma_{\tau_1}}d\zeta\int_{\gamma_{\rho_2}(1)}dw\int_{\gamma_{\tau_2}}d\omega
\frac{\zeta^{n-i}H_{\Delta n,\Delta m,\Delta a}(w)}{H_{N+1-j,\Delta m,\Delta a}(\omega)(w-\omega)(\zeta-w)}\\
:&=A_1(i,j)+I_3.\notag
\end{align}
We note that
\begin{equation}\label{Int1}
\frac 1{\pii}\int_{\gamma_{\tau_1}}\frac{d\zeta}{\zeta^{i-n}(\zeta-w)}=-\frac{1(i>n)}{w^{i-n}},
\end{equation}
since $|w|>|\zeta|$, and hence by (\ref{A2}),
\begin{equation}\label{I3}
I_3=-1(i>n)A_2(i,j).
\end{equation}
Also
\begin{equation}\label{Int2}
\frac 1{\pii}\int_{\gamma_{\tau_2}}\frac{d\omega}{\omega^{n+1-j}(z-\omega)}=\frac{1(j\le n)}{z^{n+1-j}},
\end{equation}
and we obtain
\begin{equation*}
I_2=\frac{1(j\le n)c(i)c(j)^{-1}}{(\pii)^2}\int_{\gamma_{r_1}(1)}dz\int_{\gamma_{\tau_1}}d\zeta\frac{H_{j-1,m.a}(z)(1-\zeta)}{H_{i,m,a}(\zeta)(z-\zeta)(1-z)}.
\end{equation*}
Deform $\gamma_{r_1}(1)$ to $\gamma_{\rho_1}(1)$. We then cross the pole at $z=\zeta$ and we obtain, using (\ref{A3}),
\begin{equation}\label{I2}
I_2=1(j\le n)A_3(i,j)+\frac{1(j\le n)c(i)c(j)^{-1}}{\pii}\int_{\gamma_{\tau_1}}\zeta^{i-j-1}\,d\zeta=1(j\le n)A_3(i,j)+1(i\le n)\delta_{ij}.
\end{equation}
Combining (\ref{L1step}), (\ref{I1}), (\ref{I3}), (\ref{I2}) and (\ref{Cepsilon}), we get (\ref{L1delta}).

Consider next,
\begin{align*}
L_2(i,j)&=-\frac{c(i)c(j)^{-1}}{(\pii)^4}\int_{\gamma_{r_3}(1)}dz\int_{\gamma_{\tau_1}}d\zeta\int_{\gamma_{r_2}(1)}dw\int_{\gamma_{\tau_2}}d\omega\left(\sum_{x=0}^{\infty}
\left(\frac{1-z}{1-w}\right)^x\right)\\
&\times\frac{H_{n,m,a}(z)H_{\Delta n,\Delta m,\Delta a}(w)(1-\zeta)}{H_{i,m,a}(\zeta)H_{N+1-j,\Delta m,\Delta a}(\omega)(z-\zeta)(w-\omega)
(1-z)(1-w)},
\end{align*}
where now $r_2>r_3>1+\max(\rho_1,\rho_2)$. Thus,
\begin{equation*}
\sum_{x=0}^{\infty}\left(\frac{1-z}{1-w}\right)^x=\frac{1-w}{z-w},
\end{equation*}
and consequently
\begin{align*}
L_2(i,j)&=-\frac{c(i)c(j)^{-1}}{(\pii)^4}\int_{\gamma_{r_3}(1)}dz\int_{\gamma_{\tau_1}}d\zeta\int_{\gamma_{r_2}(1)}dw\int_{\gamma_{\tau_2}}d\omega
\\&\times \frac{H_{n,m,a}(z)H_{\Delta n,\Delta m,\Delta a}(w)(1-\zeta)}{H_{i,m,a}(\zeta)H_{N+1-j,\Delta m,\Delta a}(\omega)(z-\zeta)
(w-\omega)(z-w)(1-z)}.
\end{align*}
We now deform $\gamma_{r_3}(1)$ to $\gamma_{\rho_3}(1)$, and doing so we pass the pole at $z=\zeta$, and find
\begin{align*}
L_2(i,j)&=-\frac{c(i)c(j)^{-1}}{(\pii)^4}\int_{\gamma_{\rho_3}(1)}dz\int_{\gamma_{\tau_1}}d\zeta\int_{\gamma_{r_2}(1)}dw\int_{\gamma_{\tau_2}}d\omega
\\&\times \frac{H_{n,m,a}(z)H_{\Delta n,\Delta m,\Delta a}(w)(1-\zeta)}{H_{i,m,a}(\zeta)H_{N+1-j,\Delta m,\Delta a}(\omega)(z-\zeta)
(w-\omega)(z-w)(1-z)}\\
&-\frac{c(i)c(j)^{-1}}{(\pii)^3}\int_{\gamma_{\tau_1}}d\zeta\int_{\gamma_{r_2}(1)}dw\int_{\gamma_{\tau_2}}d\omega
\frac{H_{\Delta n,\Delta m,\Delta a}(w)}{\zeta^{i-n}H_{N+1-j,\Delta m,\Delta a}(\omega)
(w-\omega)(\zeta-w)}:=J_1+J_2.
\end{align*}
In $J_1$ we deform $\gamma_{r_2}(1)$ to $\gamma_{\rho_2}(1)$. Since $\rho_2>\rho_3$, we only cross the pole at $w=\omega$, and we get
\begin{align*}
J_1&=-B_1(i,j)-\frac{c(i)c(j)^{-1}}{(\pii)^3}\int_{\gamma_{\rho_3}(1)}dz\int_{\gamma_{\tau_1}}d\zeta\int_{\gamma_{\tau_2}}d\omega
\frac{H_{n,m,a}(z)(1-\zeta)}{\omega^{n+1-j}H_{i,m,a}(\zeta)(z-\zeta)(z-\omega)(1-z)}\\:&=-B_1(i,j)+J_3.
\end{align*}
Using (\ref{Int1}), we find
\begin{align*}
J_2&=\frac{1(i>n)c(i)c(j)^{-1}}{(\pii)^2}\int_{\gamma_{r_2}(1)}dw\int_{\gamma_{\tau_2}}d\omega\frac{H_{N-i,\Delta m,\Delta a}(w)}{H_{N+1-j,\Delta m,\Delta a}(\omega)(w-\omega)}\\
&=1(i>n)A_2(i,j)+\frac{1(i>n)c(i)c(j)^{-1}}{\pii}\int_{\gamma_{\tau_2}} \omega^{j-i-1}d\omega=1(i>n)(A_2(i,j)+\delta_{ij}),
\end{align*}
which gives (\ref{L2delta}) and the lemma is proved.
\end{proof}

\begin{proof}[Proof of Lemma \ref{LemBlockdet}]
We start with the right side of (\ref{detidentity}),
\begin{align*}
&\det(I+F_{u})_{L^2(\Lambda,\rho)}=\sum_{k=0}^{\infty}\frac 1{k!}\int_{\Lambda^k}d\rho^k(\lambda)\det(F_{u}(\lambda_p,\lambda_q))_{1\le p,q\le k}\\
&=\sum_{k=0}^{\infty}\frac 1{k!}\sum_{r_1,\dots,r_k=1}^2\int_{\mathbb{R}^k}d\nu_{r_1}(x_1)\dots d\nu_{r_k}(x_k)
\det\left(M_{u}(n+[x_p]+1,n+[x_q]+1)\right)_{1\le p,q\le k}\\
&=\sum_{k=0}^{\infty}\frac 1{k!}\sum_{i_1=-n}^{N-n-1}\dots\sum_{i_k=-n}^{N-n-1}\det\left(M_{u,v}(n+i_p+1,n+i_q+1)\right)_{1\le p,q\le k}\\
&=\det(\delta_{ij}+M_{u})(i,j))_{1\le i,j\le N},
\end{align*}
where we recall that $M_{u}(i,j)=0$ if $i,j\notin\{1,\dots,N\}$.
\end{proof}

\begin{proof}[Proof of Lemma \ref{LemRescaled}]

By the formula (\ref{paAformula5}) for $P(a;A)$ and Lemma \ref{LemBlockdet}, we see that
\begin{align}\label{paArescaled}
P(a;A)&=\frac{c_3(t_1T)^{1/3}}{\pii}\int_{\gamma_r}\frac{1}{u-1}\det(I+F_{u})_{L^2(\lambda,\rho)}du\\
&=\frac{1}{\pii}\int_{\gamma_r}\frac{du}{u-1}\det(I+F_{u})_{L^2(\lambda,\rho)}du.\notag
\end{align}
We have the Fredholm expansion,
\begin{equation}\label{Fredexp}
\det\left(I+F_{u}\right)_{L^2(\lambda,\rho)}=
\sum_{k=0}^\infty\frac 1{k!}\sum_{r_1,\dots,r_k=1}^2\int_{\mathbb{R}^k}d\nu_{r_1}(x_1)\dots d\nu_{r_k}(x_k)\det\left(F_{u}(r_p,x_p;r_q,x_q)\right)_{1\le p,q\le k}.
\end{equation}
The change of variables $x_p\to c_0(t_1T)^{1/3}x_p$ gives
$$d\nu_{r_p}(c_0(t_1T)^{1/3}x_p)=c_0(t_1T)^{1/3}d\nu_{r_p}(x_p).$$
Take the factor $c_0(t_1T)^{1/3}$ into row $p$. We see then that the right side of (\ref{Fredexp}) equals,
\begin{equation*}
\sum_{k=0}^\infty\frac 1{k!}\sum_{r_1,\dots,r_k=1}^2\int_{\mathbb{R}^k}d\nu_{r_1}(x_1)\dots d\nu_{r_k}(x_k)\det\left(\tilde{F}_{u}(r_p,x_p;r_q,x_q)\right)_{1\le p,q\le k}
=\det\left(I+\tilde{F}_{u}\right)_{L^2(\lambda,\rho)}.
\end{equation*}
Combining this with (\ref{paArescaled}) we have proved the lemma.
\end{proof}

We want to prove that the operator $K(u)$ in the definition of the two-time distribution is a trace-class operator.

\begin{lemma}\label{lemtraceclass}
The operator $K(u)$ defined by (\ref{Kuv}) is a trace-class operator on the space $X$ given by (\ref{Xspace}).
\end{lemma}
\begin{proof}
Write
\begin{equation*}
S_2^*(x,y)=1(x>0)S_2(x,y),\quad S_3^*(x,y)=S_3(x,y)1(y<0)
\end{equation*}
so that 
\begin{equation*}
S=S_1-S_2^*+S_3^{*},\quad T=-T_1+S_2^{*}-S_3^*.
\end{equation*}
By splitting $K(u)$ into several parts and factoring out multiplicative constants, we see that it is enough to prove that
\begin{equation*}
\begin{pmatrix} A & A \\A & A  \end{pmatrix}
\end{equation*}
is a trace-class operator on $X$ for $A=S_1, T_1, S_2^*,S_3^*$. We can think of $A$ as an operator on $L^2(\Lambda,\rho)$ instead, where
$\Lambda=\{1,2\}\times\mathbb{R}$ and $\rho$ is given by (\ref{rhomeasure}).

Define the kernels
\begin{align}\label{HSfunctions}
&a_1(x,s)=S_3(x,s)e^{-\delta s},\quad a_2(s,y)=e^{\delta s}S_2(s,y),\\
&b_1(x,s)=\alpha 1(x>0)e^{-(\delta-\alpha\Delta\eta)x}\Ai(\Delta\xi+\Delta\eta^2+\alpha x+s),\notag\\
&b_2(x,s)=e^{(\delta-\alpha\Delta\eta)y}\Ai(\Delta\xi+\Delta\eta^2+\alpha y+s),\notag\\
&c_1(x,s)=e^{-(\delta-\eta_1)x}\Ai(\xi_1+\eta_1^2-x+s),\quad
c_2(x,s)=e^{(\delta-\eta_1)y}\Ai(\xi_1+\eta_1^2-y+s)1(y<0).\notag
\end{align}
Using the definitions, we see that
\begin{align}\label{STformulas}
S_1(x,y)&=\int_0^\infty(-a_1(x,s))a_2(s,y)\,ds,\quad T_1(x,y)=\int_{-\infty}^0a_1(x,s)a_2(s,y)\,ds\\
S_2^*(x,y)&=\int_0^\infty b_1(x,s))b_2(s,y)\,ds,\quad S_3^*(x,y)=\int_0^\infty c_1(x,s)c_2(s,y)\,ds,\notag
\end{align}
To get kernels on $L^2(\Lambda,\rho)$, we define
\begin{align*}
a_1(r_1,x;1,s)&=b_1(r_1,x;1,s)=c_1(r_1,x;1,s)=0\\
a_2(1,s;r_3,y)&=b_2(1,s;r_3,y)=c_2(1,s;r_3,y)=0.
\end{align*}
for $r_1=1,2$, and
\begin{equation*}
\tilde{a}_1(r_1,x;2,s)=\tilde{a}_2(2,s;r_3,y)=0
\end{equation*}
for $r_1=1,2$. Furthermore, we define
\begin{align*}
-a_1(r_1,x;2,s)&=\tilde{a}_1(r_1,x;1,s)=a_1(x,s)\\
a_2(2,s;r_3,y)&=\tilde{a}_2(2,s;r_3,y)=a_2(s,y)\\
b_1(r_1,x;2,s)&=b_1(x,s),\quad b_2(2,s;r_3,y)=b_2(s,y)\\
c_1(r_1,x;2,s)&=c_1(x,s),\quad c_2(2,s;r_3,y)=c_2(s,y).
\end{align*}
Then, by (\ref{STformulas}) and (\ref{rhomeasure}),
\begin{equation*}
\int_{\Lambda}a_1(r_1,x;r_2,z)a_2(r_2,z;r_3,y)\,d\rho(r_2,z)=S_1(r_1,x;r_2,y),
\end{equation*}
so $S_1=a_1a_2$. Similarly, we see that $T_1=\tilde{a}_1\tilde{a}_2$, $S_2^*=b_1b_2$ and $S_3^*=c_1c_2$. Using (\ref{deltacondition}) and asymptotic properties of the Airyfunction, we see that $a_1,a_2,b_1,b_2, c_1, c_2$ are square integrable over $\mathbb{R}^2$, and also over $\mathbb{R}$ if we fix one of the variables to be zero.
It follows from this that $a_1, a_2, \tilde{a}_1,\dots, c_2$ are Hilbert-Schmidt operators on $L^2(\Lambda,\rho)$. Since the composition of two Hilbert-Schmidt operators is
a trace-class operator, we have that $S_1, T_1, S_2^*$ and $S_3^{*}$ are trace-class operators on $L^2(\Lambda,\rho)$, and hence $K(u)$ is a trace-class
operator also.
\end{proof}

\section{Asymptotic analysis}\label{secasymptotics}
In this section we will prove Lemma \ref{LemScalinglimit}. The proof has several steps and we will split it into a sequence of lemmas. The proofs of these lemmas will appear later in the section.

For $k=1,2,3$, we define the rescaled kernels
\begin{align}\label{Atilde}
\tilde{A}_{1,T}(x,y)&=c_0(t_1T)^{1/3}A_1(n+[c_0(t_1T)^{1/3}x]+1,n+[c_0(t_1T)^{1/3}y]+1),\\
\tilde{A}_{2,T}(x,y)&=1(x\ge 0)c_0(t_1T)^{1/3}A_2(n+[c_0(t_1T)^{1/3}x]+1,n+[c_0(t_1T)^{1/3}y]+1),\notag\\
\tilde{A}_{3,T}(x,y)&=1(y<0)c_0(t_1T)^{1/3}A_3(n+[c_0(t_1T)^{1/3}x]+1,n+[c_0(t_1T)^{1/3}y]+1),\notag\\
\tilde{B}_{1,T}(x,y)&=c_0(t_1T)^{1/3}B_1(n+[c_0(t_1T)^{1/3}x]+1,n+[c_0(t_1T)^{1/3}y]+1).\notag
\end{align}

\begin{lemma}\label{LemA1}
Uniformly, for $x,y$ in a compact subset of $\mathbb{R}$, we have the limits
\begin{align}\label{Sklimit}
\lim_{T\to\infty}\tilde{A}_{1,T}(x,y)&=S_1(x,y),\\
\lim_{T\to\infty}\tilde{A}_{2,T}(x,y)&=1(x\ge 0)S_2(x,y)\notag\\
\lim_{T\to\infty}\tilde{A}_{3,T}(x,y)&=S_3(x,y)1(y<0),\notag
\end{align}
and
\begin{equation}\label{T1limit}
\lim_{T\to\infty}\tilde{B}_{1,T}(x,y)=T_1(x,y).
\end{equation}
\end{lemma}

The lemma is proved below. In order to prove the convergence of the Fredholm determinant we also need some estimates.

\begin{lemma}\label{LemKernelestimates}
Assume that $|\xi|,|\eta|\le L$ for some fixed $L$. If we choose $\delta$ in (\ref{conjfactor}) sufficiently large, depending on $q$ and $L$, there are positive constants
$C_0, C_1, C_2$ that only depend on $q$ and $L$, so that for all $x,y$ satisfying
\begin{equation}\label{xycondition}
0\le n+[c_0(t_1T)^{1/3}x]<N,\quad 0\le n+[c_0(t_1T)^{1/3}y]<N,
\end{equation}
we have the estimates
\begin{align}\label{xyest}
\left|\tilde{A}_{1,T}(x,y)\right|&\le C_0 e^{-C_1(-x)_+^{3/2}-C_2(x)_+-C_1(y)_+^{3/2}-C_2(-y)_+},\\
\left|\tilde{B}_{1,T}(x,y)\right|&\le C_0 e^{-C_1(-x)_+^{3/2}-C_2(x)_+-C_1(y)_+^{3/2}-C_2(-y)_+},\notag\\
\left|\tilde{A}_{2,T}(x,y)\right|&\le C_0 1(x\ge 0)e^{-C_1(x)_+^{3/2}-C_1(y)_+^{3/2}-C_2(-y)_+},\notag\\
\left|\tilde{A}_{3,T}(x,y)\right|&\le C_0 1(y<0)e^{-C_1(-x)_+^{3/2}-C_2(x)_+ -C_1(-y)_+^{3/2}}.\notag
\end{align}
Here $(x)_+=\max(0,x)$.

\end{lemma}

The proof is given below. We now have the estimates that we need to prove Lemma \ref{LemScalinglimit}

\begin{proof}[Proof of Lemma \ref{LemScalinglimit}]
Recall from (\ref{Ru}) and (\ref{Kuv}) that
\begin{equation*}
K_u(1,x;s,y)=S(x,y)+u^{-1}T(x,y),\quad K_u(2,x;s,y)=uS(x,y)+T(x,y),
\end{equation*}
$s=1,2$. It follows from Lemma \ref{LemA1} that
\begin{equation}\label{Ftildelimit}
\lim_{T\to\infty}\tilde{F}_{u,T}(r,x;s,y)=K_{u}(r,x;s,y),
\end{equation}
for $r,s\in\{1,2\}$, uniformly for $u,x,y$ in compact sets. From (\ref{xyest}) we see that for all $\xi,\eta,u$ in compact sets there are positive constants $C_0,C_1$ so that
\begin{equation}\label{Ftildeest}
\left|\tilde{F}_{u,T}(r,x;s,y)\right|\le C_0e^{-C_1(|x|+|y|)},
\end{equation}
for $r,s\in\{1,2\}$ and all $x,y\in\mathbb{R}$. Note that, by definition $\tilde{F}_{u,T}$ is zero if $x,y$ do not satisfy (\ref{xycondition}). We can expand the Fredholm determinant,
\begin{equation}\label{Fredexpt}
\det(I+\tilde{F}_{u,T})_{L^2(\Lambda,\rho)}=\sum_{k=0}^\infty\frac 1{k!}\int_{\Lambda^k}\det(\tilde{F}_{u,T}(\lambda_i, \lambda_j))_{1\le i,j\le k}d^k\rho(\lambda)
\end{equation}
in its Fredholm expansion. It follows from (\ref{Ftildelimit}), (\ref{Ftildeest}) and Hadamard's inequality that we can take the limit $T\to\infty$ in (\ref{Fredexp}) and get
\begin{equation*}
\sum_{k=0}^\infty\frac 1{k!}\int_{\Lambda^k}\det(K_u(\lambda_i,\lambda_j))_{1\le i,j\le k}d^k\rho(\lambda)=\det(I+K_u)_X.
\end{equation*}
This completes the proof.

\end{proof}

Consider
\begin{equation*}
H_{k,\ell,b}(w)=\frac{w^k(1-w)^{b+\ell}}{\left(1-\frac w{1-q}\right)^\ell}
\end{equation*}
with the scalings ($K\to\infty$, $\eta,\xi,v$ fixed),
\begin{align}\label{klbscaling}
k&=K-c_1\eta K^{2/3}+c_0vK^{1/3},\\
\ell&=K+c_1\eta K^{2/3},\notag\\
b&=c_2K+c_3\xi K^{1/3}.\notag
\end{align}
Here the constants $c_i$ are given by (\ref{scalingconstants}). Write
\begin{equation}\label{fw}
f(w)=\log H_{k,\ell,b}(w)=k\log w+(b+\ell)\log(1-w)-\ell\log(1-\frac{w}{1-q}).
\end{equation}
If $\eta=\xi=v=0$, then $f(w)$ has a double critical point at
\begin{equation}\label{wc}
w_c=1-\sqrt{q}.
\end{equation}
Define
\begin{equation}\label{Hstar}
H^*_{k,\ell,b}(w)=\frac{H_{k,\ell,b}(w)}{H_{k,\ell,b}(w_c)}.
\end{equation}
The local asymptotics around the critical point is given by the next lemma.

\begin{lemma}\label{LemHstarlimit}
Fix $L>0$ and assume that  $|\xi|,|\eta|, |v|\le L$. Furthermore, assume that we have the scaling (\ref{klbscaling}). Then, uniformly for $w'$ in a compact set in $\mathbb{C}$
\begin{equation}\label{Hstarlimit}
\lim_{K\to\infty} H^*_{k,\ell,b}\left(w_c+\frac{c_4}{K^{1/3}}w'\right)=\exp(\frac 13w'^3+\eta w'^2-(\xi-v)w'),
\end{equation}
where
\begin{equation}\label{c4}
c_4=\frac{q^{1/3}(1-\sqrt{q})}{(1+\sqrt{q})^{1/3}}.
\end{equation}

\end{lemma}

\begin{proof}
Let
\begin{align*}
f_1(w)&=\log w+(c_2+1)\log(1-w)-\log\left(1-\frac w{1-q}\right),\\
f_2(w)&=-\log w+\log(1-w)-\log\left(1-\frac w{1-q}\right),\\
f_3(w)&=c_0x\log w+c_3\xi\log(1-w),
\end{align*}
so that
\begin{equation}\label{Hnmx2}
f(w)=Kf_1(w)+c_1\eta K^{2/3}f_2(w)+K^{1/3}f_3(w).
\end{equation}
Then $f_1'(w)$ has a double zero at $w_c$
only if the constant $c_2=2\sqrt{q}/(1-\sqrt{q})$. A computation gives
\begin{equation*}
f_1^{(3)}(w_c)=\frac{2(1+\sqrt{q})}{q(1-\sqrt{q})^3},
\end{equation*}
and we find
\begin{equation}\label{ftaylor}
K\left(f_1\left(w_c+\frac{c_4}{K^{1/3}}w'\right)-f_1(w_c)\right)=\frac 13 w'^3+O\left(\frac{|w'|^4}{K^{1/3}}\right).
\end{equation}
Also, 
\begin{equation}\label{g1taylor}
c_1\eta K^{2/3}\left(f_2\left(w_c+\frac{c_4}{K^{1/3}}w'\right)-f_2(w_c)\right)=\eta w'^2+O\left(\frac{|w'|^3}{K^{1/3}}\right),
\end{equation}
and
\begin{equation}\label{g2taylor}
K^{2/3}\left(f_3\left(w_c+\frac{c_4}{K^{1/3}}w'\right)-f_3(w_c)\right)=-(\xi-v)w'+O\left(\frac{|w'|^2}{K^{1/3}}\right).
\end{equation}
Using (\ref{ftaylor}), (\ref{g1taylor}) and (\ref{g2taylor}) in (\ref{Hnmx2}), we obtain
\begin{equation}\label{Hstarnmx}
H^*_{k,\ell,b}\left(w_c+\frac{c_4}{K^{1/3}}w'\right)=
\exp\left(\frac 13w'^3+\eta w'^2-(\xi-x)w'+O(|w'|^4/K^{1/3})\right)\notag
\end{equation}
as $K\to\infty$. 

\end{proof}

To prove the estimates that we need, we use some explicit contours in (\ref{A1}) to (\ref{A3}). Let $d>0$ and define
\begin{equation}\label{w1}
w_1(\sigma)=w_1(\sigma;d)=w_c(1-\frac d{K^{1/3}})e^{\mathrm{i}\sigma/K^{1/3}},
\end{equation}
and
\begin{equation}\label{w2}
w_2(\sigma)=w_2(\sigma;d)=1- \sqrt{q}(1-\frac d{K^{1/3}})e^{\mathrm{i}\sigma/K^{1/3}},
\end{equation}
for $|\sigma|\le\pi K^{1/3}$, where $K$ is as in (\ref{klbscaling}). Thus, $w_1$ gives a circle around the origin of radius $w_c(1-\frac d{K^{1/3}})$, and $w_2$ gives a circle of 
radius $ \sqrt{q}(1-\frac d{K^{1/3}})$ around $1$.

\begin{lemma}\label{LemHstarest1}
Fix $L>0$. Assume that we have the scaling (\ref{klbscaling}) and that $|\xi|,|\eta|, |v|\le L$. Then, there are positive constants $C_j$, $1\le j\le 4$ that only depend on $q$ and $L$,
so that if $C_1\le d\le C_2$, then
\begin{equation}\label{Hstarest1}
\left|H^*_{k,\ell,b}(w_1(\sigma;d))\right|^{-1}\le C_3e^{-C_4\sigma^2},
\end{equation}
and
\begin{equation}\label{Hstarest2}
\left|H^*_{k,\ell,b}(w_2(\sigma;d))\right|\le C_3e^{-C_4\sigma^2},
\end{equation}
for $|\sigma|\le \pi K^{1/3}$.

\end{lemma}

We will also need estimates that work for large $v$.

\begin{lemma}\label{LemHstarest2}
Assume that $|\xi|,|\eta|\le L$ for some fixed $L>0$, and assume that we have  the scaling (\ref{klbscaling}) and $v$ is such that $k\ge 0$. Then, we can choose $d=d(v)\ge C_0$,
so that
\begin{equation}\label{Hstarest3}
\left|H^*_{k,\ell,b}(w_1(\sigma;d(v)))\right|^{-1}\le C_1e^{-C_2\sigma^2-\mu_1(-v)_+^{3/2}+\mu_2(v)_+},
\end{equation}
for $|\sigma|\le \pi K^{1/3}$, where $C_0,C_1,C_2,\mu_1,\mu_2$ are positive constants that only depend on $q$ and $L$. Similarly, there is a choice of $d=d(v)$ so that
\begin{equation}\label{Hstarest4}
\left|H^*_{k,\ell,b}(w_2(\sigma;d(v)))\right|\le C_1e^{-C_2\sigma^2-\mu_1(-v)_+^{3/2}+\mu_2(v)_+}.
\end{equation}

\end{lemma}

These two Lemmas will be proved below. We can use Lemma \ref{LemHstarlimit} and Lemma \ref{LemHstarest1} to prove Lemma \ref{LemA1}.

\begin{proof}[Proof of Lemma \ref{LemA1}]
It follows from (\ref{conjfactor}), (\ref{A1}) and (\ref{Hstar}) that
\begin{align}\label{Atildestar}
\tilde{A}_{1,T}(i,j)&=\frac{c_0(t_1T)^{1/3}e^{-\delta(x-y)}}{(\pii)^4}\int_{\gamma_{\rho_1}(1)}dz\int_{\gamma_{\rho_2}(1)}dw\int_{\gamma_{\tau_1}}d\zeta\int_{\gamma_{\tau_2}}d\omega
\\&\times\frac{H^*_{n,m,a}(z)H^*_{\Delta n,\Delta m,\Delta a}(w)(1-\zeta)(1-\sqrt{q})^{-1}}{H^*_{n+[c_0(t_1T)^{1/3}x]+1,m,a}(\zeta)H^*_{\Delta n-[c_0(t_1T)^{1/3}y],\Delta m,\Delta a}(\omega)(z-\zeta)(w-\omega)(z-w)(1-z)}.\notag
\end{align}

Let $\Gamma_D$ denote the vertical line through $D$ oriented upwards, $\mathbb{R}\ni t\mapsto D+\mathrm{i}t$. Let $D_1>D_2>0$, $d_1,d_2>0$ be such that
\begin{equation*}
C_1\le\frac{c_4}{\sqrt{q}}D_r\le C_2,\quad C_1\le\frac{c_4}{\sqrt{q}}d_r\le C_2,
\end{equation*}
$r=1,2$, where $C_1,C_2$ are the constants in Lemma \ref{LemHstarest1} with some fixed $L$ arbitrarily large. We choose the following parametrizations in (\ref{Atildestar}),
\begin{equation}\label{paramet1}
z(\sigma_1)=w_2\left(\frac{c_4\sigma_1}{\sqrt{q}},\frac{c_4D_1}{\sqrt{q}}\right),\quad \zeta(\sigma_3)=w_1\left(\frac{c_4\sigma_3}{\sqrt{q}},\frac{c_4d_1}{\sqrt{q}}\right),
\end{equation}
where $K=K_1=(t_1T)^{1/3}$ in (\ref{w1}), (\ref{w1}), and
\begin{equation}\label{paramet2}
w(\sigma_2)=w_2\left(\frac{c_4\sigma_2}{\sqrt{q}},\frac{c_4D_2}{\sqrt{q}}\right),\quad \omega(\sigma_4)=w_1\left(\frac{c_4\sigma_4}{\sqrt{q}},\frac{c_4d_2}{\sqrt{q}}\right),
\end{equation}
where $K=K_2=(\Delta tT)^{1/3}$, 
\begin{equation}\label{sigmai}
|\sigma_i|\le\pi K_1^{1/3}, \text{ for $i=1,3$}, \quad |\sigma_i|\le\pi K_2^{1/3}, \text{ for $i=2,4$}.
\end{equation}
Recall the condition (\ref{radii}) on the radii. 
Let
\begin{align}\label{hfunctions}
h_1(\sigma_1)=H^*_{n,m,a}(z(\sigma_1))&,\quad h_2(\sigma_2)=H*_{\Delta n,\Delta m,\Delta a}(w(\sigma_2)),\\
h_3(\sigma_3)=H^*_{n+[c_0(t_1T)^{1/3}x]+1,m,a}(\zeta(\sigma_3))&,\quad h_4(\sigma_4)=H^*_{\Delta n-[c_0(t_1T)^{1/3}y],\Delta m,\Delta a}(\omega(\sigma_4)).
\notag
\end{align}
Now, a computation shows that, for some constant $C$,
\begin{equation}\label{estimate}
\left|\frac{c_0K_1^{1/3}}{(z(\sigma_1)-\zeta(\sigma_3))(w(\sigma_2)-\omega(\sigma_4))(z(\sigma_1)-w(\sigma_2))}\frac{dz}{d\sigma_1}\frac{dw}{d\sigma_2}
\frac{d\zeta}{d\sigma_3}\frac{d\omega}{d\sigma_4}\right|\le C
\end{equation}
for all $\sigma_i$ satisfying (\ref{sigmai}). Thus, for $x,y$ in a compact set, we have the following bound on the integrand in (\ref{Atildestar}),
\begin{align}\label{intbound}
&\left|\frac{c_0K_1^{1/3}h_1(\sigma_1)h_2(\sigma_2)(1-\zeta(\sigma_3))(1-z(\sigma_1))^{-1}}{h_3(\sigma_3)h_4(\sigma_4)
(z(\sigma_1)-\zeta(\sigma_3))(w(\sigma_2)-\omega(\sigma_4))(z(\sigma_1)-w(\sigma_2))}\frac{dz}{d\sigma_1}\frac{dw}{d\sigma_2}
\frac{d\zeta}{d\sigma_3}\frac{d\omega}{d\sigma_4}\right|\\
&\le C\left|\frac{h_1(\sigma_1)h_2(\sigma_2)}{h_3(\sigma_3)h_4(\sigma_4)}\right|\le C_3'e^{-C_4'(\sigma_1^2+\sigma_2^2+\sigma_3^2+\sigma_4^2)},\notag
\end{align}
where the last inequality follows from Lemma \ref{LemHstarest1}.

For $\sigma_i$ in a bounded set, we see that
\begin{align}\label{approxparamet}
z(\sigma_1)&=w_c+\frac{c_4}{K_1^{1/3}}(-\mathrm{i}\sigma_1+D_1)+O(K_1^{-2/3}),\\
w(\sigma_2)&=w_c+\frac{c_4}{K_2^{1/3}}(-\mathrm{i}\sigma_2+D_2)+O(K_2^{-2/3}),\notag\\
\zeta(\sigma_3)&=w_c+\frac{c_4}{K_1^{1/3}}(\mathrm{i}\sigma_3+d_1)+O(K_1^{-2/3}),\notag\\
\omega(\sigma_4)&=w_c+\frac{c_4}{K_2^{1/3}}(\mathrm{i}\sigma_4+d_2)+O(K_2^{-2/3}),\notag
\end{align}
It follows from (\ref{scaling}) that
\begin{align}
n&=K_1-c_1\eta_1K_1^{2/3},\quad \Delta n=K_2-c_1\Delta\eta K_2^{2/3}\\
m&=K_1-c_1\eta_2K_1^{2/3},\quad \Delta m=K_2+c_1\Delta\eta K_2^{2/3}\notag\\
a&=c_2K_1+c_3\xi_1K_1^{1/3}, \quad \Delta a=c_2K_2+c_3\Delta\xi K_2^{1/3},\notag
\end{align}
and hence
\begin{align}
n+c_0x(t_1T)^{1/3}&=K_1-c_1\eta_1 K_1^{2/3}+c_0xK_1^{1/3},\\
\Delta n-c_0y(t_1T)^{1/3}&=K_2-c_1\Delta\eta K_1^{2/3}-c_0\alpha yK_2^{1/3}.\notag
\end{align}
Write $z'=-\mathrm{i}\sigma_1+D_1$, $w'=-\mathrm{i}\sigma_2+D_2$, $\zeta'=\mathrm{i}\sigma_3+d_1$, $\omega'=\mathrm{i}\sigma_4+d_2$. Note that
\begin{align}\label{differentials}
&c_0(t_1T)^{1/3}\frac{dzdwd\zeta d\omega}{(z-\zeta)(w-\omega)(z-w)}=\alpha (1-\sqrt{q})\frac{dz'dw'd\zeta' d\omega'}{(z'-\zeta')(w'-\omega')(z'-\alpha w')},
\\
&c_0(t_1T)^{1/3}\frac{dzd\zeta}{z-\zeta}=(1-\sqrt{q})\frac{dz'd\zeta'}{z'-\zeta'},\quad
c_0(t_1T)^{1/3}\frac{dwd\omega}{w-\omega}=\alpha (1-\sqrt{q})\frac{dw'd\omega'}{w'-\omega'}.\notag
\end{align}

It follows from Lemma \ref{LemHstarlimit}, (\ref{Atildestar}), (\ref{approxparamet}), (\ref{intbound}) and the dominated convergence theorem
that
\begin{align}\label{A1limit}
\lim_{T\to\infty}\tilde{A}_{1,T}(x,y)&=\frac{\alpha e^{\delta(y-x)}}{(\pii)^4}\int_{\Gamma_{D_1}}dz'\int_{\Gamma_{D_2}}dw'
\int_{\Gamma_{-d_1}}d\zeta'\int_{\Gamma_{-d_2}}d\omega'\\
&\times\frac{e^{\frac 13z'^3+\eta_1z'^2-\xi_1z'+\frac 13w'^3+\Delta\eta w'^2-\Delta\xi w'}}{e^{\frac 13\zeta'^3+\eta_1\zeta'^2-(\xi_1-x)\zeta'
+\frac 13\omega'^3+\Delta\eta \omega'^2-(\Delta\xi+\alpha y)\omega'}
(z'-\zeta')(w'-\omega')(z'-\alpha w')},\notag
\end{align}
and we have the condition
\begin{equation}\label{Gammacondition}
d_1,d_2>0,\quad 0<D_1<\alpha D_2<D_3.
\end{equation}

Define
\begin{equation}\label{Gxieta}
G_{\xi,\eta}(z)=e^{\frac 13 z^3+\eta z^2-\xi z},
\end{equation}
and let
\begin{equation}\label{S1intformula}
S_1(x,y)=\frac{\alpha e^{\delta(y-x)}}{(\pii)^4}\int_{\Gamma_{D_1}}dz\int_{\Gamma_{D_2}}dw
\int_{\Gamma_{-d_1}}d\zeta\int_{\Gamma_{-d_2}}d\omega\frac{G_{\xi_1,\eta_1}(z)G_{\Delta\xi,\Delta\eta}(w)}
{G_{\xi_1-x,\eta_1}(\zeta)G_{\Delta\xi+\alpha y,\Delta\eta}(\omega)(z-\zeta)(w-\omega)(z-\alpha w)}.
\end{equation}
If $d,D>0$, we have the formulas,
\begin{align}\label{Gint}
\frac 1{\pii}\int_{\Gamma_D}G_{\xi,\eta}(z)\,dz&=\Ai(\xi+\eta^2)e^{\xi\eta+\frac 23\eta^3},\\
\frac 1{\pii}\int_{\Gamma_{-d}}\frac{d\zeta}{G_{\xi,\eta}(\zeta)}&=\Ai(\xi+\eta^2)e^{-\xi\eta-\frac 23\eta^3},
\notag
\end{align}
with absolutely convergent integrals. Using (\ref{Gammacondition}), we see that
\begin{equation*}
\frac 1{z-\zeta}=\int_0^\infty e^{-s_1(z-\zeta)}ds_1,\quad \frac 1{w-\omega}=\int_0^\infty e^{-s_2(w-\omega)}ds_2,
\quad  \frac 1{z-\alpha w}=-\int_0^\infty e^{s_3(z-\alpha w)}ds_3.
\end{equation*}
It follows from these formulas, (\ref{S1intformula}) and (\ref{Gint}) that $S_1$ is also given by (\ref{S1}).

The proof of (\ref{T1limit}) is identical with $D_1$ replaced by $D_3$ satisfying 
(\ref{Gammacondition}). The integral formula for $T_1$ reads
\begin{equation}\label{T1intformula}
T_1(x,y)=\frac{\alpha e^{\delta(y-x)}}{(\pii)^4}\int_{\Gamma_{D_3}}dz\int_{\Gamma_{D_2}}dw
\int_{\Gamma_{-d_1}}d\zeta\int_{\Gamma_{-d_2}}d\omega\frac{G_{\xi_1,\eta_1}(z)G_{\Delta\xi,\Delta\eta}(w)}
{G_{\xi_1-x,\eta_1}(\zeta)G_{\Delta\xi+\alpha y,\Delta\eta}(\omega)(z-\zeta)(w-\omega)(z-\alpha w)}.
\end{equation}
The other cases are treated similarly. For $S_2$ and $S_3$ we get the formulas
\begin{equation}\label{S2intformula}
S_2(x,y)=\frac{\alpha e^{\delta(y-x)}}{(\pii)^2}\int_{\Gamma_{D_2}}dw\int_{\Gamma_{-d_2}}d\omega\frac{G_{\Delta\xi+\alpha x,\Delta\eta}(w)}
{G_{\Delta\xi+\alpha y,\Delta\eta}(\omega)(w-\omega)},
\end{equation}
and
\begin{equation}\label{S3intformula}
S_3(x,y)=\frac{e^{\delta(y-x)}}{(\pii)^2}\int_{\Gamma_{D_1}}dz\int_{\Gamma_{-d_1}}d\zeta\frac{G_{\xi_1-y,\eta_1}(\zeta)}{G_{\xi_1-x,\eta_1}(\zeta)(z-\zeta)}.
\end{equation}

This proves Lemma \ref{LemA1}.

\end{proof}

\begin{proof}[Proof of Lemma \ref{LemKernelestimates}]
Consider first $\tilde{A}_{1,T}$. By Lemma (\ref{LemHstarest1}), we can choose $d_1$ and $d_2$, with $d_1<\alpha d_2$, so that
\begin{align}\label{Est1}
|H^*_{n,m,a}(w_2(\sigma_1,d_1))|&\le C_3e^{-C_4\sigma_1^2}, \quad |\sigma_1|\le\pi K_1^{1/3},\\
|H^*_{\Delta n,\Delta m,\Delta a}(w_2(\sigma_2,d_1))|&\le C_3e^{-C_4\sigma_1^2}, \quad |\sigma_2|\le\pi K_2^{1/3},\notag
\end{align}
where $C_3,C_4$ are some positive constants independent of $\sigma_1$ and $\sigma_2$. By Lemma \ref{LemHstarest2}, we can choose $d=d_3(x)\ge C_0$,
and $d=d_4(y)\ge C_0$, so that
\begin{align}\label{Est2}
|H^*_{n+[c_0x(t_1T)^{1/3}]+1,m,a}(w_1(\sigma_3,d_3(x)))|^{-1}&\le C_1e^{-C_2\sigma_3^2-\mu_1(-x)_+^{3/2}+\mu_2(x)_+},\\
|H^*_{\Delta n-[c_0y(t_1T)^{1/3}],\Delta m,\Delta a}(w_1(\sigma_4,d_4(y)))|^{-1}&\le C_1e^{-C_2\sigma_3^2-\mu_1(y)_+^{3/2}+\mu_2(-y)_+},\notag
\end{align}
It is not difficult to check that if $z=w_2(\sigma_1,d_1)$, $w=w_2(\sigma_2,d_2)$, $\zeta=w_1(\sigma_3,d_3(x))$ and $\omega=w_1(\sigma_4,d_4(y))$, then there is a
constant $C_5$ so that
\begin{equation*}
|z-\zeta|\ge C_5K_1^{-1/3}, \quad |w-\omega|\ge C_5K_2^{-1/3},
\end{equation*}
and
\begin{equation*}
|z-w|\ge \sqrt{q}|d_1-\alpha d_2|K_1^{-1/3}\ge C_5K_1^{-1/3}.
\end{equation*}
Introducing these parametrizations into (\ref{Atildestar}) and using the estimates above, we find
\begin{align}
|\tilde{A}_{1,T}(x,y)|&\le Ce^{-\delta(x-y)-\mu_1(-x)_+^{3/2}+\mu_2(x)_+-\mu_1(y)_+^{3/2}+\mu_2(-y)_+}\int_{\mathbb{R}^4}
e^{-C_4(\sigma_1^2+\sigma_2^2+\sigma_3^2+\sigma_4^2)}d^4\sigma\\
&\le Ce^{-\delta(x-y)-\mu_1(-x)_+^{3/2}+\mu_2(x)_+-\mu_1(y)_+^{3/2}+\mu_2(-y)_+}.\notag
\end{align}
We see that for large enough $|x|$, we can choose $\delta$ so large that
\begin{equation*}
-\mu_1(-x)_+^{3/2}+\mu_2(x)_+-\delta x\le -C_1(-x)_+^{3/2}-C_2(x)_+
\end{equation*}
for some positive constants $C_1,C_2$. This proves the estimate for $\tilde{A}_{1,T}$. The proof for $\tilde{B}_{1,T}$ is completely analogous.

Consider now $\tilde{A}_{3,T}$,
\begin{equation}\label{A3tildestar}
\tilde{A}_{3,T}(x,y)=\frac{c_0(t_1T)^{1/3}e^{-\delta(x-y)}1(y<0)}{(2\pii)^2}\int_{\gamma_{\rho_1}(1)}dz\int_{\gamma_{\tau_1}}d\zeta
\frac{H^*_{n+[c_0y(t_1T)^{1/3}],m,a}(z)(1-\zeta)}{H^*_{n+[c_0x(t_1T)^{1/3}]+1,m,a}(z)(1-z)(z-\zeta)}.
\end{equation}
Using Lemma \ref{LemHstarest2}, we see that, just as for $\tilde{A}_{1,T}$, we can choose $d_1(y)$ and $d_2(x)$ so that
\begin{align*}
|H^*_{n+[c_0y(t_1T)^{1/3}],m,a}(w_2(\sigma_1,d_1(y)))|&\le C_1e^{-C_2\sigma_1^2-\mu_1(-y)_+^{3/2}+\mu_2(y)_+},\\
|H^*_{n+[c_0x(t_1T)^{1/3}],m,a}(w_1(\sigma_2,d_2(x)))^{-1}|&\le C_1e^{-C_2\sigma_1^2-\mu_1(-x)_+^{3/2}+\mu_2(x)_+},
\end{align*}
and we get the estimate
\begin{equation*}
|\tilde{A}_{3,T}(x,y)|\le Ce^{-\mu_1(-x)_+^{3/2}+\mu_2(x)_+-\delta x-\mu_1(-y)_+^{3/2}+\delta y}1(y<0).
\end{equation*}
This gives us the estimate we want by choosing $\delta$ large enough. The proof for $\tilde{A}_{2,T}$ is analogous.

\end{proof}

The statements in Lemma \ref{LemHstarest1} and in Lemma \ref{LemHstarest2} are consequences of two other lemmas that we will now state and prove.
The first lemma is concerned with the decay along the paths given by $w_1(\sigma)$ and $w_2(\sigma)$.

\begin{lemma}\label{LemFirstest}
Assume that we have the scaling (\ref{klbscaling}) and let $|\xi|,|\eta|\le L$ for some fixed $L>0$.
There are positive constants $C_1,C_2,C_3,C_4$ that only depend on $q$ and $L$, so that if
\begin{equation}\label{dcondition}
C_1\le d\le C_2K^{1/3}
\end{equation}
then for $|\sigma|\le\pi K^{1/3}$,
\begin{equation}\label{Hklbest1}
\left|\frac{H_{k,\ell,b}(w_1(\sigma;d))}{H_{k,\ell,b}(w_1(0;d))}\right|^{-1}\le C_3 e^{-C_4d\sigma^2},
\end{equation}
for all $v\in\mathbb{R}$. Furthermore, for $|\sigma|\le\pi K^{1/3}$,
\begin{equation}\label{Hklbest2}
\left|\frac{H_{k,\ell,b}(w_2(\sigma;d))}{H_{k,\ell,b}(w_1(0;d))}\right|\le C_3 e^{-C_4d\sigma^2},
\end{equation}
for all $v\le 0$ such that $k\ge 0$, and all $v$ such that $|v|\le L$.
\end{lemma}

\begin{proof}
Recall the definition of $f(w)$ in (\ref{fw}) and the parametrizations (\ref{w1}) and (\ref{w2}). Define
\begin{equation}\label{gr}
g_r(\sigma)=\re f(w_r(\sigma))=
k\log|w_r(\sigma)|+(b+\ell)\log|1-w_r(\sigma)|-\ell\log\left|1-\frac{w_r(\sigma)}{1-q}\right|,
\end{equation}
$r=1,2$, $|\sigma|\le\pi K^{1/3}$. Note that for any real numbers $\alpha,\beta$,
\begin{equation}\label{dsigma}
\frac{d}{d\sigma}\log|1-\alpha e^{\mathrm{i}\beta\sigma}|=\frac{\alpha\beta\sin{\beta\sigma}}{(1-\alpha)^2+4\alpha\sin^2(\beta\sigma/2)}.
\end{equation}
Let $\beta=K^{-1/3}$, $\alpha_1=w_c(1-dK^{-1/3})$, $\alpha_2=\alpha_1/(1-q)$. Then a computation using (\ref{gr}) and (\ref{dsigma}) gives
\begin{equation}\label{g1prime}
g_1'(\sigma)=\frac{(b+\ell)\alpha_1(1-\alpha_2)^2-\ell\alpha_2(1-\alpha_1)^2+4b\alpha_1\alpha_2\sin^2\frac{\beta\sigma}{2}}
{((1-\alpha_1)^2+4\alpha_1\sin^2\frac{\beta\sigma}{2})((1-\alpha_2)^2+4\alpha_2\sin^2\frac{\beta\sigma}{2})}\beta\sin\beta\sigma.
\end{equation}
By symmetry it is enough to consider $0\le \sigma\le\pi K^{1/3}$. We have to compute
\begin{equation}\label{alphaexpr}
(b+\ell)\alpha_1(1-\alpha_2)^2-\ell\alpha_2(1-\alpha_1)^2=\frac{\alpha_1}{1-q}[(1-q)(b+\ell)(1-\alpha_2)^2-\ell(1-\alpha_1)^2].
\end{equation}
Now,
\begin{equation*}
1-\alpha_1=\sqrt{q}+(1-\sqrt{q})d\beta,\quad 1-\alpha_2=\frac 1{1+\sqrt{q}}(\sqrt{q}+d\beta),
\end{equation*}
and using (\ref{klbscaling}) a computation gives
\begin{align*}
&(1-q)(b+\ell)(1-\alpha_2)^2-\ell(1-\alpha_1)^2\\&=\left(2qd-\frac{2c_1q^{3/2}}{1+\sqrt{q}}\eta\right)K^{2/3}+\left(\sqrt{q}d^2-\frac{2c_1q(1-\sqrt{q})}{1+\sqrt{q}}\eta d+
\frac{c_3(1-\sqrt{q})q}{1+\sqrt{q}}\xi\right)K^{1/3}\\
&-\frac{c_1\sqrt{q}(1-\sqrt{q})}{1+\sqrt{q}}\eta d^2+\frac{2c_3\sqrt{q}(1-\sqrt{q})}{1+\sqrt{q}}\xi d+\frac{c_3(1-\sqrt{q})}{1+\sqrt{q}}\xi d^2K^{-1/3}.
\end{align*}
Since $|\xi|,|\eta|\le L$, we see that
\begin{equation}\label{alphaest}
(1-q)(b+\ell)(1-\alpha_2)^2-\ell(1-\alpha_1)^2\ge qdK^{2/3}+\Delta_1K^{2/3}+\Delta_2K^{1/3},
\end{equation}
where
\begin{align*}
\Delta_1&=qd-\frac{2c_1q^{3/2}}{1+\sqrt{q}}L,\\
\Delta_2&=\sqrt{q}d^2-\frac{2c_1q(1-\sqrt{q})}{1+\sqrt{q}}Ld-\frac{c_3(1-\sqrt{q})q}{1+\sqrt{q}}L-\frac{c_1\sqrt{q}(1-\sqrt{q})}{1+\sqrt{q}}Ld^2K^{-1/3}\\
&-\frac{2c_3\sqrt{q}(1-\sqrt{q})}{1+\sqrt{q}}LdK^{-1/3}-\frac{c_3(1-\sqrt{q})}{1+\sqrt{q}}L d^2K^{-2/3}.
\end{align*}
We note that we can choose $C_1$ and $C_2$, depending only on $q$ and $L$, so that if $C_1\le d\le C_2K^{1/3}$, then
$\Delta_1\ge 0$ and $\Delta_2\ge 0$, and also
\begin{equation*}
\frac{\alpha_1}{1-q}\ge\frac{w_c}{2(1-q)}=\frac 1{2(1+\sqrt{q})}.
\end{equation*}
Thus, we see from (\ref{alphaexpr}) and (\ref{alphaest}) that
\begin{equation*}
(b+\ell)\alpha_1(1-\alpha_2)^2-\ell\alpha_2(1-\alpha_1)^2\ge\frac{q}{2(1+\sqrt{q})}dK^{2/3}
\end{equation*}
provided that $C_1\le d\le C_2K^{1/3}$. Consequently, by (\ref{g1prime}),
\begin{equation}\label{g1primeestt}
g_1'(\sigma)\ge\frac{qdK^{2/3}\sin K^{-2/3}\sigma}{2(1+\sqrt{q})(1+\alpha_1)^2(1+\alpha_2)^2}\ge \frac{q}{8(1+\sqrt{q})}dK^{2/3}\sin K^{-2/3}\sigma
\end{equation}
since
\begin{equation*}
(1+\alpha_1)^2(1+\alpha_2)^2\le 4.
\end{equation*}
It follows, by integration, that, for $0\le\sigma\le\pi K^{1/3}$,
\begin{equation*}
g_1(\sigma)-g_1(0)\ge \frac{q}{4(1+\sqrt{q})}dK^{4/3}\sin^2\left(\frac{\sigma}{2K^{2/3}}\right)\ge\frac{q}{4(1+\sqrt{q})}dK^{4/3}\left(\frac{2\sigma}{2\pi K^{2/3}}\right)^2
=\frac{q}{4\pi^2(1+\sqrt{q})}d\sigma^2,
\end{equation*}
since by convexity $\sin t\ge 2t/\pi$ for $0\le t\le \pi/2$.
This proves the estimate (\ref{Hklbest1}).

Next, we turn to the proof of (\ref{Hklbest2}) which is similar. In this case we get
\begin{equation*}
g_2'(\sigma)=\frac{d}{d\sigma}\left(k\log|1-\sqrt{q}(1-d\beta)e^{\mathrm{i}\beta\sigma}|-\ell\log|1-\frac 1{\sqrt{q}}(1-d\beta)e^{\mathrm{i}\beta\sigma}|\right),
\end{equation*}
where $\beta=K^{-1/3}$. Let $\alpha_1=\sqrt{q}(1-d\beta)$, $\alpha_2=\frac 1q \alpha$. Then, using (\ref{dsigma}), we obtain
\begin{equation}\label{g2prime}
g_2'(\sigma)=\frac{k\alpha_1(1-\alpha_2)^2-\ell\alpha_2(1-\alpha_1)^2+4(k-\ell)\alpha_1\alpha_2\sin^2\frac{\beta\sigma}{2}}
{((1-\alpha_1)^2+4\alpha_1\sin^2\frac{\beta\sigma}{2})((1-\alpha_2)^2+4\alpha_2\sin^2\frac{\beta\sigma}{2})}\beta\sin\beta\sigma.
\end{equation}
Now,
\begin{equation}\label{alphaexpr2}
k\alpha_1(1-\alpha_2)^2-\ell\alpha_2(1-\alpha_1)^2=\frac{\alpha_1}q\left[kq(1-\alpha_2)^2-\ell(1-\alpha_1)^2\right],
\end{equation}
and a computation gives
\begin{equation*}
kq(1-\alpha_2)^2-\ell(1-\alpha_1)^2=-3(1-\sqrt{q})dK^{2/3}-\Delta K^{2/3},
\end{equation*}
where
\begin{align}\label{Delta}
\Delta&=(1-\sqrt{q})d+2c_1(1-\sqrt{q})^2\eta-(1-q)dK^{-1/3}-c_0(1-\sqrt{q})^2vK^{-1/3}\\
&+2c_1(1+q)\eta dK^{-2/3}+2c_0(1-\sqrt{q})vdK^{-2/3}-c_0vd^2K^{-1}.\notag
\end{align}
If $|\xi|,|\eta|,|v|\le L$, we see that we can choose $C_1, C_2$, depending only on $q,L$, so that if $C_1\le d\le C_2K^{1/3}$, the $\Delta\ge 0$, and we obtain
\begin{equation}\label{alphaest2}
kq(1-\alpha_2)^2-\ell(1-\alpha_1)^2\le-3(1-\sqrt{q})dK^{2/3}.
\end{equation}
If $|\xi|,|\eta|\le L$ and $v\le 0$, we can also choose $C_1, C_2$ so that $\Delta\ge 0$ if $C_1\le d\le C_2K^{1/3}$. Also, we see that
\begin{align}\label{alphaest3}
4(k-\ell)\alpha_1\alpha_2\sin^2\frac{\beta\sigma}{2}&=\left(-2c_1\eta K^{2/3}+c_0vK^{1/3}\right)\alpha_1\alpha_2\sin^2\frac{\sigma}{2K^{1/3}}\\
&\le 8(c_0+c_1)L\alpha_1\alpha_2K^{2/3}\le 8(c_0+c_1)LK^{2/3}.\notag
\end{align}
if $v\le 0$ or $|v|\le L$. Assume that $C_2$ is such that $\alpha_1\ge\sqrt{q}/2$. Then (\ref{alphaexpr2}), (\ref{alphaest2}) and (\ref{alphaest3}) give
\begin{align*}
&k\alpha_1(1-\alpha_2)^2-\ell\alpha_2(1-\alpha_1)^2+4(k-\ell)\alpha_1\alpha_2\sin^2\frac{\beta\sigma}{2}\\&\le
-\frac 1{\sqrt{q}}(1-\sqrt{q})dK^{2/3}+\left(-\frac{1-\sqrt{q}}{2\sqrt{q}}d+8(c_0+c_1)L\right)K^{2/3}\\
&\le -\frac 1{\sqrt{q}}(1-\sqrt{q})dK^{2/3},
\end{align*}
if we choose $C_1$ so that
\begin{equation*}
-\frac{1-\sqrt{q}}{2\sqrt{q}}d+8(c_0+c_1)L\le 0
\end{equation*}
for $d\ge C_1$. Since $\alpha_1\le\sqrt{q}$, $\alpha_1\le 1/\sqrt{q}$,
\begin{equation*}
\frac 1{(1+\alpha_1)^2(1+\alpha_2)^2}\ge \frac 1{(2+\sqrt{q}+1/\sqrt{q})^2},
\end{equation*}
and (\ref{g2prime}) gives
\begin{equation*}
g_2'(\sigma)ß\le -\frac{1-\sqrt{q}}{\sqrt{q}(2+\sqrt{q}+1/\sqrt{q})^2}dK^{2/3}.
\end{equation*}
We can now proceed, as for $g_1$, to prove that
\begin{equation*}
g_2(\sigma)-g_2(0)\le -\frac{1-\sqrt{q}}{\pi^2\sqrt{q}(2+\sqrt{q}+1/\sqrt{q})^2}d\sigma^2.
\end{equation*}
This completes the proof of the Lemma.

\end{proof}

The next Lemma is concerned with the decay for large $|v|$.

\begin{lemma}\label{LemSecest}
Assume that we have the scaling (\ref{klbscaling}) and that $v$ is such that $k\ge 0$, which will always be the case. Also, assume that $|\xi|,|\eta|\le L$ for some $L>0$. There are positive
constants $\mu_1,\mu_2,\mu_3$ that only depend on $q,L$, and a choice $d=d(v)$ satisfying (\ref{dcondition}) so that
\begin{equation}\label{Hklbest3}
\left|\frac{H_{k,\ell,b}(w_c)}{H_{k,\ell,b}(w_1(0;d(v)))}\right|\le\mu_3e^{-\mu_1(-v)_+^{3/2}+\mu_2(v)_+}.
\end{equation}
There is also a choice $d=d(v)$ satisfying (\ref{dcondition}) so that
\begin{equation}\label{Hklbest4}
\left|\frac{H_{k,\ell,b}(w_2(0;d(v)))}{H_{k,\ell,b}(w_c)}\right|\le\mu_3e^{-\mu_1(-v)_+^{3/2}+\mu_2(v)_+}.
\end{equation}
If we assume that $|v|\le L$, we can choose $d$ independent of $v$ in some interval so that (\ref{Hklbest3}) and (\ref{Hklbest4}) hold.
\end{lemma}

\begin{proof}
Using (\ref{gr}) we see that
\begin{equation*}
\left|\frac{H_{k,\ell,b}(w_1(0;d(v)))}{H_{k,\ell,b}(w_c)}\right|=e^{g_1(0)-\log f(w_c)},
\end{equation*}
so we want to estimate $g_1(0)-\log f(w_c)$ from below, and then make a good choice of $d$. We see that
\begin{equation}\label{g10}
g_1(0)-\log f(w_c)=k\log(1-dK^{-1/3})+(b+\ell)\log\left(1+\frac{1-\sqrt{q}}{\sqrt{q}}dK^{-1/3}\right)-\ell\log\left(1+\frac 1{\sqrt{q}}dK^{-1/3}\right).
\end{equation}
To estimate this expression, we will use the inequalities
\begin{equation}\label{logineq}
-x-\frac{x^2}2-\frac{2x^3}{3}\le\log(1-x)\le -x-\frac{x^2}2-\frac{x^3}3,
\end{equation}
for $1/2\le x\le 1$, and
\begin{equation}\label{logineq2}
x-\frac{x^2}2\le\log(1+x)\le x-\frac{x^2}2+\frac{x^3}3,
\end{equation}
for $x\ge 0$.
It follows from (\ref{g10}) and these inequalities that
\begin{align*}
g_1(0)-\log f(w_c)&\ge k\left(-dK^{-1/3}-\frac 12d^2K^{-2/3}-\frac 23d^3K^{-1}\right)\\&+(b+\ell)\left(\frac{1-\sqrt{q}}{\sqrt{q}}dK^{-1/3}-\frac 12\left(\frac{1-\sqrt{q}}{\sqrt{q}}\right)^2d^2K^{-2/3}
\right)\\
&+\ell\left(-\frac 1{\sqrt{q}}dK^{-1/3}+\frac 1{2q}d^2K^{-2/3}-\frac 1{3q^{3/2}}d^3K^{-1}\right)
\end{align*}
Substitute the expressions in (\ref{klbscaling}). After some manipulation this gives
\begin{align}\label{g10est}
g_1(0)-\log f(w_c)&\ge\left(-c_0v+\frac{1-\sqrt{q}}{\sqrt{q}}c_3\xi\right)d+
\left(\frac 1{\sqrt{q}}c_1\eta-\frac12c_0vK^{-1/3}-\frac{(1-\sqrt{q})^2}{2q}c_3\xi K^{-1/3}\right)d^2\\
&+\left(-\frac 23-\frac 1{3q^{3/2}}+\left(\frac 23-\frac 1{3q^{3/2}}\right)c_1\eta K^{-1/3}-\frac 23c_0vK^{-2/3}\right)d^3\notag\\
&\ge\left(-c_0v-\frac{1-\sqrt{q}}{\sqrt{q}}c_3L\right)d+\left(-\frac 1{\sqrt{q}}c_1L-\frac12c_0vK^{-1/3}-\frac{(1-\sqrt{q})^2}{2q}c_3LK^{-1/3}\right)d^2\notag\\
&+\left(-\frac 23-\frac 1{3q^{3/2}}-\left|\frac 23-\frac 1{3q^{3/2}}\right|c_1LK^{-1/3}-\frac 23c_0vK^{-2/3}\right)d^3.\notag
\end{align}
If $|v|\le L$, we see that if we choose $d$ so that $C_1'\le d\le C_2'$, then
\begin{equation*}
g_1(0)-\log f(w_c)\ge -C_3'.
\end{equation*}
Here $C_1',C_2', C_3'$ only depend on $q,L$.
If $v\le 0$, then it follows from (\ref{g10est}) that
\begin{align}\label{g10est2}
g_1(0)-\log f(w_c)&\ge
\left(-c_0v-\frac{1-\sqrt{q}}{\sqrt{q}}c_3L\right)d+\left(-\frac 1{\sqrt{q}}c_1L-\frac{(1-\sqrt{q})^2}{2q}c_3LK^{-1/3}\right)d^2\\
&+\left(-\frac 23-\frac 1{3q^{3/2}}-\left|\frac 23-\frac 1{3q^{3/2}}\right|c_1LK^{-1/3}\right)d^3.\notag
\end{align}
Choose $d=\epsilon\sqrt{-v}$. Then, by (\ref{g10est2}),
\begin{align}\label{g10est3}
g_1(0)-\log f(w_c)&\ge
c_0\epsilon(-v)^{3/2}\left[1-\left(\frac{1-\sqrt{q}}{\sqrt{q}}\right)\frac{c_3L}{-v}-\left(\frac 1{\sqrt{q}}c_1L+\frac{(1-\sqrt{q})^2}{2q}\frac{c_3L}{K^{1/3}}\right)\epsilon^2\frac 1{\sqrt{-v}}\right.\\
&-\left.\left(\frac 23+\frac 1{3q^{3/2}}+\left|\frac 23-\frac 1{3q^{3/2}}\right|\frac{c_1L}{K^{1/3}}\right)\epsilon^2\right] \notag
\end{align}
Choose $D_1$ large, depending on only $q,L$, so that
\begin{equation*}
\left(\frac{1-\sqrt{q}}{\sqrt{q}}\right)\frac{c_3L}{-v}\le \frac14,\quad \left(\frac 1{\sqrt{q}}c_1L+\frac{(1-\sqrt{q})^2}{2q}\frac{c_3L}{K^{1/3}}\right)\frac 1{\sqrt{-v}}\le 1,
\end{equation*}
if $\sqrt{-v}\ge D_1$. Since $k\ge 0$, there is a constant $D_2$ so that $\sqrt{-v}\le D_2K^{1/3}$. The condition (\ref{dcondition}) becomes
\begin{equation*}
\frac{C_1}{\sqrt{-v}}\le\epsilon\le\frac{C_2K^{1/3}}{\sqrt{-v}},
\end{equation*}
which is satisfied if
\begin{equation}\label{epsiloncondition}
\frac{C_1}{D_1}\le\epsilon\le\frac{C_2}{D_2}.
\end{equation}
We can choose $D_1$ so large that $C_1/D_1$ is as small as we want, and hence we can choose $\epsilon$ so small that
\begin{equation*}
\left(1+\frac 23+\frac 1{3q^{3/2}}+\left|\frac 23-\frac 1{3q^{3/2}}\right|\frac{c_1L}{K^{1/3}}\right)\epsilon^2\le \frac14.
\end{equation*}
It then follows from (\ref{g10est3}) that 
\begin{equation*}
g_1(0)-\log f(w_c)\ge\frac 12c_0\epsilon(-v)^{3/2}
\end{equation*}
for $\sqrt{-v}\ge D_1$. By adjusting $\mu_3$, we see that (\ref{Hklbest3}) holds if $v\le 0$.

If $v\ge 0$, we choose a $d$ satistying (\ref{dcondition}) depending on $q,L$, but not on $v$ or $K$. It follows from (\ref{g10est2}) that there are constants $\mu_1$
and $\mu_3'$, so that
\begin{equation*}
g_1(0)-\log f(w_c)\ge -\mu_1(v)_+-\mu_3'.
\end{equation*}
Hence (\ref{Hklbest3}) holds also when $v\ge 0$.

To prove (\ref{Hklbest4}) we consider instead
\begin{align*}
g_2(0)-\log f(w_c)&=k\log\left(1+\frac{\sqrt{q}}{1-\sqrt{q}}dK^{-1/3}\right)+(b+\ell)\log(1-dK^{-1/3})-\ell\log\left(1-\frac 1{1-\sqrt{q}}dK^{-1/3}\right)\\
&\le k\left(\frac{\sqrt{q}}{1-\sqrt{q}}dK^{-1/3}-\frac{q}{2(1-\sqrt{q})^2}d^2K^{-2/3}+\frac{q^{3/2}}{(1-\sqrt{q})^3}d^3K^{-1}\right)\\&+
(b+\ell)\left(-dK^{-1/3}-\frac 12d^2K^{-2/3}-\frac 13d^3K^{-1}\right)\\
&+\ell\left(\frac{1}{1-\sqrt{q}}dK^{-1/3}+\frac 1{2(1-\sqrt{q})^2}d^2K^{-2/3}+\frac{2}{(1-\sqrt{q})^3}d^3K^{-1}\right),
\end{align*}
by (\ref{logineq}) and (\ref{logineq2}). Into this estimate we insert the expressions in (\ref{klbscaling}), and after some computation we get
\begin{align*}
&g_2(0)-\log f(w_c)\le\left(\frac{\sqrt{q}}{1-\sqrt{q}}c_0v-c_3\xi\right)d\\
&+\frac 1{2(1-\sqrt{q})^2}\left(2\sqrt{q}c_1\eta-qc_0vK^{-1/3}+c_3(1-\sqrt{q})^2\xi K^{-1/3}\right)d^2\\
&+\frac 1{3(1-\sqrt{q})^3}\left(1+\sqrt{q}+q+(1+3\sqrt{q}-3q)c_1\eta K^{-1/3}+q^{3/2}c_0vK^{-2/3}-c_3(1-\sqrt{q})^3\xi K^{-2/3}\right)d^3.
\end{align*}
We can now proceed in analogy with the previous case to show (\ref{Hklbest4}).

\end{proof}

\section{More formulas for the two-time distribution}\label{SecFormulas}
In this section we give an alternative formula for the two-time distribution, see Proposition \ref{propQ} below.

Recall the notation (\ref{Gxieta}),
\begin{equation}\label{Gxieta2}
G_{\xi,\eta}(z)=e^{\frac 13 z^3+\eta z^2-\xi z}.
\end{equation}
Looking at (\ref{Gint}), we see that it is natural to write
\begin{equation}\label{Aixieta}
\Ai_{\xi,\eta}(x,y)=\Ai(\xi+\eta^2+x+y)e^{(\xi+x+y)\eta+\frac 23\eta^3},
\end{equation}
since we then get the formulas
\begin{align}\label{GAiformulas}
\frac 1{\pii}\int_{\Gamma_D}G_{\xi+x+y,\eta}(z)\,dz&=\Ai_{\xi,\eta}(x,y),\\
\frac 1{\pii}\int_{\Gamma_{-d}}\frac{d\zeta}{G_{\xi+x+y,\eta}(\zeta)}&=\Ai_{\xi,-\eta}(x,y),
\end{align}
for any $d,D>0$. We can think of (\ref{Aixieta}) as the kernel of an integral operator on $L^2(\mathbb{R}_+)$.

In order to give a different formula for the two-time distribution, we need to define several kernels. We will write
\begin{equation}\label{alphaprime}
\alpha'=(1+\alpha^3)^{1/3}=\left(\frac{t_2}{\Delta t}\right)^{1/3}.
\end{equation}
Let
\begin{align}\label{M1}
M_1(v_1,v_2)&=\frac{e^{\delta(v_1-v_2)}}{(\pii)^2}\int_{\Gamma_D}dz\int_{\Gamma_{-d}}d\zeta\frac{G_{\xi_1+v_1,\eta_1}(z)}{G_{\xi_1+v_2,\eta_1}(\zeta)(z-\zeta)}\\
&=e^{\delta(v_1-v_2)}\int_0^\infty\Ai_{\xi_1,\eta_1}(v_1,\lambda)\Ai_{\xi_1,-\eta_1}(\lambda,v_2)\,d\lambda,\notag
\end{align}
\begin{align}\label{M2}
M_2(v_1,v_2)&=\frac{1}{(\pii)^2\alpha'}\int_{\Gamma_D}dz\int_{\Gamma_{-d}}d\zeta\frac{G_{\xi_2+v_2/\alpha',\eta_2}(z)}{G_{\xi_2+v_1/\alpha',\eta_2}(\zeta)(z-\zeta)}\\
&=\frac 1{\alpha'}\int_0^\infty\Ai_{\xi_2,-\eta_2}(v_1/\alpha',\lambda)\Ai_{\xi_2,\eta_2}(\lambda,v_2/\alpha')\,d\lambda,\notag
\end{align}
and
\begin{align}\label{M3}
M_3(v_1,v_2)&=\frac{1}{(\pii)^2}\int_{\Gamma_D}dz\int_{\Gamma_{-d}}d\zeta\frac{G_{\Delta\xi+v_2,\Delta\eta}(z)}{G_{\Delta\xi+v_1,\Delta\eta}(\zeta)(z-\zeta)}\\
&=\int_0^\infty\Ai_{\Delta\xi,-\Delta\eta}(v_1,\lambda)\Ai_{\Delta\xi,\Delta\eta_1}(\lambda,v_2)\,d\lambda,\notag
\end{align}

We will also need the following kernels. Let
\begin{equation}\label{dDconditions}
0<d_1<\alpha d_2<d_3,\quad 0<D_1<\alpha D_2<D_3.
\end{equation}
Define
\begin{align}\label{k1def}
&k_1(v_1,v_2)\\&=\frac{\alpha}{(\pii)^4}\int_{\Gamma_{D_3}}dz\int_{\Gamma_{D_2}}dw\int_{\Gamma_{-d_3}}d\zeta\int_{\Gamma_{-d_2}}d\omega
\frac{G_{\xi_1,\eta_1}(z)G_{\Delta\xi+v_2,\Delta\eta}(w)}{G_{\xi_1,\eta_1}(\zeta)G_{\Delta\xi+v_1,\Delta\eta}(\omega)(z-\zeta)(z-\alpha w)(\alpha\omega-\zeta)}\notag\\
&=\alpha\int_{\mathbb{R}_+^3}\Ai_{\Delta\xi,-\Delta\eta}(v_1,-\alpha\lambda_1)\Ai_{\xi_1,-\eta_1}(\lambda_1,\lambda_2)\Ai_{\xi_1,\eta_1}(\lambda_2,\lambda_3)
\Ai_{\Delta\xi,\Delta\eta}(-\alpha\lambda_3,v_2)\,d^3\lambda,\notag
\end{align}
\begin{align}\label{k2def}
&k_2(v_1,v_2)\\&=\frac{\alpha}{(\pii)^3}\int_{\Gamma_{D_3}}dz\int_{\Gamma_{D_2}}dw\int_{\Gamma_{-d_2}}d\omega
\frac{G_{\xi_1,\eta_1}(z)G_{\Delta\xi+v_2,\Delta\eta}(w)}{G_{\xi_2+v_1/\alpha',\eta_2}(\omega)(\alpha'z-\alpha\omega)(z-\alpha w)}\notag\\
&=\alpha\int_{\mathbb{R}_+^2}\Ai_{\xi_2,-\eta_2}(\frac{v_1}{\alpha'},\alpha\lambda_1)\Ai_{\xi_1,\eta_1}(\alpha'\lambda_1,\lambda_2)
\Ai_{\Delta\xi,\Delta\eta}(-\alpha\lambda_2,v_2)\,d^2\lambda,\notag
\end{align}
\begin{align}\label{k3def}
k_3(v_1,v_2)&=\frac{\alpha e^{-\delta v_2}}{(\pii)^2}\int_{\Gamma_{-d_3}}d\zeta\int_{\Gamma_{-d_2}}d\omega
\frac 1{G_{\xi_1+v_2,\eta_1}(\zeta)G_{\Delta\xi+v_1,\Delta\eta}(\omega)(\alpha\omega-\zeta)}\\
&=\alpha e^{-\delta v_2}\int_{\mathbb{R}_+}\Ai_{\Delta\xi,-\Delta\eta}(v_1,-\alpha\lambda)\Ai_{\xi_1,-\eta_1}(\lambda,v_2)\,d\lambda,\notag
\end{align}
\begin{align}\label{k4def}
k_4(v_1,v_2)&=\frac{\alpha e^{-\delta v_2}}{\alpha'\pii}\int_{\Gamma_{-d_2}}\frac{d\omega}{G_{\xi_2+(v_1+\alpha v_2)/\alpha',\eta_2}(\omega)}\\
&=e^{-\delta v_2}\frac {\alpha}{\alpha'}\Ai_{\xi_2,-\eta_2}\left(\frac{v_1}{\alpha'},\frac{\alpha v_2}{\alpha'}\right),\notag
\end{align}
\begin{align}\label{k5def}
&k_5(v_1,v_2)\\&=\frac{\alpha}{(\pii)^3}\int_{\Gamma_{D_2}}dw\int_{\Gamma_{-d_3}}d\zeta\int_{\Gamma_{-d_2}}d\omega
\frac{G_{\xi_2+v_2/\alpha',\eta_2}(w)}{G_{\xi_1,\eta_1}(\zeta)G_{\Delta\xi+v_1,\Delta\eta}(\omega)(\alpha w-\alpha'\zeta)(\alpha\omega-\zeta)}\notag\\
&=\alpha\int_{\mathbb{R}_+^2}\Ai_{\Delta\xi,-\Delta\eta}(v_1,-\alpha\lambda_1)\Ai_{\xi_1,-\eta_1}(\lambda_1,\alpha'\lambda_2)
\Ai_{\xi_2,\eta_2}(\alpha\lambda_2,\frac{v_2}{\alpha'})\,d^2\lambda,\notag
\end{align}
\begin{align}\label{k6def}
&k_6(v_1,v_2)\\&=\frac{e^{\delta v_1}}{(\pii)^4}\int_{\Gamma_{D_3}}dz_1\int_{\Gamma_{D_1}}dz_2\int_{\Gamma_{D_2}}dw\int_{\Gamma_{-d_1}}d\zeta
\frac{G_{\xi_1,\eta_1}(z_1)G_{\xi_1+v_1,\eta_1}(z_2)G_{\Delta\xi+v_2,\Delta\eta}(w)}{G_{\xi_1,\eta_1}(\zeta)(z_1-\zeta)(z_2-\zeta)(z_1-\alpha w)}\notag\\
&=e^{\delta v_1}\int_{\mathbb{R}_+^3}\Ai_{\xi_1,\eta_1}(v_1,\lambda_1)\Ai_{\xi_1,-\eta_1}(\lambda_1,\lambda_2)\Ai_{\xi_1,\eta_1}(\lambda_2,\lambda_3)
\Ai_{\Delta\xi,\Delta\eta}(-\alpha\lambda_3,v_2)\,d^3\lambda,\notag
\end{align}
and
\begin{align}\label{k7def}
&k_7(v_1,v_2)\\&=\frac{e^{\delta v_1}}{(\pii)^3}\int_{\Gamma_{D_1}}dz\int_{\Gamma_{D_2}}dw\int_{\Gamma_{-d_1}}d\zeta
\frac{G_{\xi_1+v_1,\eta_1}(z)G_{\xi_2+v_2/\alpha',\eta_2}(w)}{G_{\xi_1,\eta_1}(\zeta)(\alpha w-\alpha'\zeta)(z-\zeta)}\notag\\
&=e^{\delta v_1}\int_{\mathbb{R}_+^2}\Ai_{\xi_1,\eta_1}(v_1,\alpha'\lambda_1)\Ai_{\xi_1,-\eta_1}(\lambda_1,\alpha'\lambda_2)
\Ai_{\xi_2,\eta_2}(\alpha\lambda_2,\frac{v_2}{\alpha'})\,d^2\lambda.\notag
\end{align}
The kernels $M_i$ and $k_i$ depend on the parameters $\alpha, \xi_1, \Delta\xi, \eta_1, \Delta\eta$ and $\delta$. When we need to indicate this dependence we write
$M_i(\alpha,  \xi_1, \Delta\xi, \eta_1, \Delta\eta, \delta)$ and $k_i(\alpha,  \xi_1, \Delta\xi, \eta_1, \Delta\eta, \delta)$. 
We then think of $\xi_2$ and $\eta_2$ as functions of $\alpha$, $\xi_1$ and $\Delta\xi$, and
$\alpha$, $\eta_1$ and $\Delta\eta$ respectively. Explicitly,
\begin{equation}\label{xi2}
\xi_2=\xi_2(\alpha, \xi_1, \Delta\xi)=\frac 1{\alpha'}(\alpha\xi_1+\Delta\xi),
\end{equation}
\begin{equation}\label{eta2}
\eta_2=\eta_2(\alpha, \eta_1, \Delta\eta)=\frac 1{\alpha'^2}(\alpha^2\eta_1+\Delta\eta).
\end{equation}

Let
\begin{equation}\label{Yspace}
Y=L^2(\mathbb{R}_+)\oplus L^2(\mathbb{R}_+)
\end{equation}
On $Y$, we define a matrix operator kernel $Q(u)$ by
\begin{equation}\label{Qu}
Q(u)=\begin{pmatrix} Q_{11}(u) &   Q_{12}(u)\\
                                             Q_{21}(u) & Q_{22}(u)
                                              
                    \end{pmatrix},
 \end{equation}
 where
 \begin{align}\label{Qdef}
 Q_{11}(u)&=(2-u-u^{-1})k_1+(u-1)(k_2+k_5)+(u-1)M_3-uM_2\\
 Q_{12}(u)&=(u+u^{-1}-2)k_3+(1-u)k_4\notag\\
 Q_{21}(u)&=(1-u^{-1})k_6-k_7\notag\\
 Q_{22}(u)&=(u^{-1}-1)M_1.\notag
 \end{align}
We will write  $Q(u,\alpha,  \xi_1, \Delta\xi, \eta_1, \Delta\eta, \delta)$ to indicate the dependence on all parameters.

\begin{proposition}\label{propQ}
The two-time distribution (\ref{ftt}) is given by
 \begin{equation}\label{Fnew1}
F_{\text{two-time}}(\xi_1,\eta_1;\xi_2,\eta_2;\alpha)=\frac{1}{\pii}\int_{\gamma_r}\frac 1{u-1}\det(I+Q(u))_Y\,du,
\end{equation}
where $r>1$.
\end{proposition}
We will give the proof below. The formula (\ref{Fnew1}) is suitable for investigating the limit $\alpha\to 0$ (long time separation). For more on this limit see \cite{DJL}.
To study the limit $\alpha\to\infty$ (short time separation), we can use (\ref{Fnew1}) and the next Proposition which gives an $\alpha$ and $1/\alpha$ relation. Let
\begin{equation}\label{betadef}
\beta=\frac 1{\alpha},\quad \beta'=(1+\beta^3)^{1/3}=\frac{\alpha'}{\alpha}.
\end{equation}
To indicate the dependence of the kernel $K(u)$ on all parameters we write $K(u,\alpha,  \xi_1, \Delta\xi, \eta_1, \Delta\eta, \delta)$.

\begin{proposition}\label{propshortlongrelation}
We have the formula
 \begin{equation}\label{Fshortlong}
F_{\text{two-time}}(\xi_1,\eta_1;\xi_2,\eta_2;\alpha)=\frac{1}{\pii}\int_{\gamma_r}\frac 1{u-1}\det(I+K(u^{-1},\beta,\Delta\xi,\xi_1,\Delta\eta,\eta_1,\delta))_X\,du,
\end{equation}
where $r>1$.
\end{proposition}

The proof is given below. Recall that
\begin{equation}\label{Deltaxieta2}
\Delta\xi=\alpha'\xi_2-\alpha\xi_1,\quad \Delta\eta=\alpha'^2\eta_2-\alpha^2\eta_1.
\end{equation}
Combining the two Propositions above we see that
 \begin{equation}\label{Fnew2}
F_{\text{two-time}}(\xi_1,\eta_1;\xi_2,\eta_2;\alpha)=\frac{1}{\pii}\int_{\gamma_r}\frac 1{u-1}\det(I+Q(u^{-1},\beta,\Delta\xi,\xi_1,\Delta\eta,\eta_1,\delta))_Y\,du.
\end{equation}
Note that $\alpha$ is replaced by $\beta=1/\alpha$, $\xi_1$ and $\Delta\xi$, as well as $\eta_1$ and $\Delta\eta$, are interchanged, and $u$ is replaced by $u^{-1}$.
This formula is suitable for studying the limit $\alpha\to\infty$ since this corresponds to $\beta\to 0$, see \cite{DJL}. Note that combining (\ref{xi2}), (\ref{eta2}) and 
(\ref{Deltaxieta2}), we get
\begin{equation}\label{xi2eta2}
\xi_2=\xi_2(\beta,\Delta\xi,\xi_1),\quad \eta_2=\eta_2(\beta,\Delta\eta,\eta_1).
\end{equation}
We now turn to the proofs of the Propositions.

\begin{proof}[Proof of Proposition \ref{propQ}]
Define the kernels
\begin{align}\label{pqkernels}
p_1(x,v)&=-\frac{e^{-\delta x}}{(\pii)^3}\int_{\Gamma_{D_3}}dz\int_{\Gamma_{D_2}}dw\int_{\Gamma_{-d_1}}d\zeta
\frac{G_{\xi_1,\eta_1}(z)G_{\Delta\xi+v,\Delta\eta}(w)}{G_{\xi_1-x,\eta_1}(\zeta)(z-\zeta)(z-\alpha w)},\\
p_2(x,v)&=-\frac{e^{-\delta x}1(x>0)}{\pii}\int_{\Gamma_{D_2}}G_{\Delta\xi+\alpha x+v,\Delta\eta}(w)\,dw,\notag\\
p_3(x,v)&=\frac{e^{-\delta (x+v)}}{\pii}\int_{\Gamma_{-d_1}}\frac{d\zeta}{G_{\xi_1+v-x,\eta_1}(\zeta)},\notag\\
p_4(x,v)&=-\frac{e^{-\delta x}}{(\pii)^2}\int_{\Gamma_{D_2}}dw\int_{\Gamma_{-d_1}}d\zeta
\frac{G_{\xi_2+v/\alpha',\eta_2}(\alpha' w)}{G_{\xi_1-x,\eta_1}(\zeta)(\alpha w-\zeta)},\notag\\
q_1(v,y)&=\frac{\alpha e^{\delta y}}{\pii}\int_{\Gamma_{-d_2}}\frac{d\omega}{G_{\Delta\xi+\alpha y+v,\Delta\eta}(\omega)},\notag\\
q_2(v,y)&=\frac{e^{\delta (y+v)}1(y<0)}{\pii}\int_{\Gamma_{D_1}}G_{\xi_1+v-y,\eta_1}(z)\,dz.\notag
\end{align}
The factors involving $\delta v$ have been introduced in order to get well-defined operators. We also define
\begin{equation}\label{S4intformula}
S_4(x,y)=-\frac{\alpha e^{\delta(y-x)}}{(\pii)^3}\int_{\Gamma_{D_2}}dw\int_{\Gamma_{-d_1}}d\zeta\int_{\Gamma_{-d_2}}d\omega
\frac{G_{\xi_2,\eta_2}(\alpha'w)}{G_{\xi_1-x,\eta_1}(\zeta)G_{\Delta\xi+\alpha y,\Delta\eta}(\omega)(\alpha w-\zeta)(w-\omega)}.
\end{equation}
From (\ref{S1intformula}) and (\ref{T1intformula}), we see that
\begin{equation}\label{S4S1T1}
S_4=S_1-T_1,
\end{equation}
by moving the $z$-integration contour. We then pick up a contribution from the pole at $z=\alpha w$, which gives $S_4$. It follows from
(\ref{T1intformula}), (\ref{S2intformula}), (\ref{S3intformula}), (\ref{S4intformula}) and (\ref{pqkernels}) that
\begin{align}\label{STfactor}
T_1(x,y)&=-\int_{\mathbb{R}_+}p_1(x,v)q_1(v,y)\,dv,\\
1(x>0)S_2(x,y)&=-\int_{\mathbb{R}_+}p_2(x,v)q_1(v,y)\,dv,\notag\\
S_3(x,y)1(y<0)&=\int_{\mathbb{R}_+}p_3(x,v)q_2(v,y)\,dv,\notag\\
S_4(x,y)&=\int_{\mathbb{R}_+}p_4(x,v)q_1(v,y)\,dv.\notag
\end{align}
From the definition of $R(u)$, (\ref{S4S1T1}) and (\ref{STfactor}), we see that
\begin{equation}\label{RuFormula}
R(u)(x,y)=(u^{-1}-1)\int_{\mathbb{R}_+}p_1(x,v)q_1(v,y)+p_2(x,v)q_1(v,y)+p_3(x,v)q_2(v,y)\,dv+\int_{\mathbb{R}_+}p_4(x,v)q_1(v,y)\,dv.
\end{equation}
Let $p_i^{\pm}$ be the operator from $L^2(\mathbb{R}_+)$ to $L^2(\mathbb{R}_\pm)$ with kernel $p_i(x,v)$, and 
$q_i^{\pm}$ be the operator from $L^2(\mathbb{R}_\pm)$ to $L^2(\mathbb{R}_+)$ with kernel $q_i(v,y)$. From the definition of $K(u)$ and (\ref{RuFormula}) it follows that
\begin{equation*}
K(u)=pq,
\end{equation*}
where
\begin{equation*}
p= \begin{pmatrix} &(u^{-1}-1)p_1^-+p_4^- &(u^{-1}-1)p_3^-\\
                              &(1-u)p_1^++(1-u)p_2^++up_4^+ &(1-u)p_3^+
                              \end{pmatrix}
\end{equation*}
and
\begin{equation*}
q= \begin{pmatrix} &q_1^- &q_1^+\\ &q_2^- &0
                              \end{pmatrix},
\end{equation*}
are matrix operators $p:Y\mapsto X$ and $q:X\mapsto Y$. Note that $p_2^-=q_2^+=0$.
Let
\begin{equation*}
Q(u)=qp
\end{equation*}
which gives an operator from $Y$ to itself. A straightforward computation using (\ref{pqkernels}), (\ref{k1def}) - (\ref{k7def})  and (\ref{M1}) -   (\ref{M3}) shows that
\begin{equation}
\begin{matrix}
&q_1^-p_1^-=-k_1(\alpha),  &q_1^+p_1^+=k_1(\alpha)-k_2(\alpha), &q_1^+p_2^+=-M_3, \\
&q_1^-p_3^-=k_3(\alpha),   &q_1^+p_3^+=-k_3(\alpha)+k_4(\alpha), &q_1^-p_4^-=-k_5(\alpha), \\
&q_1^+p_4^+=k_5(\alpha)-M_2(\alpha), &q_2^-p_1^-=-k_6(\alpha), &q_2^-p_4^-=-k_7(\alpha),\\
&q_2^-p_3^-=M_1. & & 
\end{matrix}
\end{equation}

From this we see that $Q(u)$ is given by (\ref{Qdef}). In these computations we use (\ref{xi2}) and (\ref{eta2}) to get $\xi_2,\eta_2$ from $\xi_1,\Delta\xi,  \eta_1,\Delta\eta$.
The Proposition now follows from
\begin{equation*}
\det(I+K(u))_X= \det(I+pq)_X=\det(I+qp)_Y=\det(I+Q(u))_Y.
\end{equation*}                  
\end{proof}

\begin{proof}[Proof of Proposition \ref{propshortlongrelation}]
To indicate the dependence of $S$, $T$ and $R(u)$ on all parameters we write $S(\alpha,\xi_1,\Delta\xi,\eta_1,\Delta\eta,\delta)$ etc. It is straightforward to check from the definitions that
\begin{equation*}
\frac 1{\alpha}S(\alpha,\xi_1,\Delta\xi, \eta_1,\Delta\eta,\delta)(\frac x{\alpha},\frac y{\alpha})=T(\beta,\Delta\xi,\xi_1,\Delta\eta,\eta_1,\beta\delta)(-y,-x),
\end{equation*}
and
\begin{equation*}
\frac 1{\alpha}T(\alpha,\xi_1,\Delta\xi, \eta_1,\Delta\eta,\delta)(\frac x{\alpha},\frac y{\alpha})=S(\beta,\Delta\xi,\xi_1,\Delta\eta,\eta_1,\beta\delta)(-y,-x).
\end{equation*}
It follows that
\begin{equation*}
\frac 1{\alpha}R(u,\alpha,\xi_1,\Delta\xi, \eta_1,\Delta\eta,\delta)(\frac x{\alpha},\frac y{\alpha})=u^{-1}R(u^{-1},\beta,\Delta\xi,\xi_1,\Delta\eta,\eta_1,\beta\delta)(-y,-x).
\end{equation*}
If we write
\begin{equation*}
\tilde{R}(u)(x,y)=R(u^{-1},\beta,\Delta\xi,\xi_1,\Delta\eta,\eta_1,\beta\delta)(x,y),
\end{equation*}
we see that
\begin{equation*}
\frac 1{\alpha}R(u)(-\frac y{\alpha},-\frac x{\alpha})=u^{-1}\tilde{R}(u^{-1}(x,y).
\end{equation*}
Let $K^*_{\alpha}(u)(x,y)=\alpha^{-1}K(\alpha^{-1}y,\alpha^{-1}x)$, and define $V:X\mapsto X$ by
\begin{equation*}
V\begin{pmatrix} f_1(x) \\ f_2(x) \end{pmatrix}
=\begin{pmatrix} f_2(-x) \\ f_1(-x) \end{pmatrix}.
\end{equation*}
Note that $V^2=I$. Since taking the adjoint and rescaling the kernel does not change the Fredholm determinant, we see that
\begin{equation*}
\det(I+K(u))_{X}=\det(I+K^*_{\alpha}(u))_X=\det(I+VK^*_{\alpha}(u)V)_X,
\end{equation*}
Using these definitions a computation shows that
\begin{equation*}
VK^*_{\alpha}(u)V=\begin{pmatrix} &\tilde{R}(u^{-1}(x,y) &\tilde{R}(u^{-1}(x,y) \\ &\tilde{R}(u^{-1}(x,y) &\tilde{R}(u^{-1}(x,y)
\end{pmatrix}
\begin{pmatrix} &I &0\\&0 &u^{-1}I \end{pmatrix}
\end{equation*}
This operator has the same determinant as
\begin{align*}
&\begin{pmatrix} &I &0\\&0 &u^{-1}I \end{pmatrix}
\begin{pmatrix} &\tilde{R}(u^{-1}(x,y) &\tilde{R}(u^{-1}(x,y) \\ &\tilde{R}(u^{-1}(x,y) &\tilde{R}(u^{-1}(x,y)
\end{pmatrix}=
\begin{pmatrix} &\tilde{R}(u^{-1}(x,y) &\tilde{R}(u^{-1}(x,y)\\ &u^{-1}\tilde{R}(u^{-1}(x,y) &u^{-1}\tilde{R}(u^{-1}(x,y)
\end{pmatrix}\\
&=K(u^{-1},\beta,\Delta\xi,\xi_1,\Delta\eta,\eta_1,\beta\delta)(x,y).
\end{align*}
Thus,
\begin{align*}
&\det(I+K(u,\alpha,\xi_1,\Delta\xi,\eta_1,\Delta\eta,\delta))_X=\det(I+K(u^{-1},\beta,\Delta\xi,\xi_1,\Delta\eta,\eta_1,\beta\delta))_X\\&=
\det(I+K(u^{-1},\beta,\Delta\xi,\xi_1,\Delta\eta,\eta_1,\delta)_X
\end{align*}
since the Fredholm determinant is independent of the value of $\delta$ as long as the condition (\ref{deltacondition}) is satisfied. Note that
this condition is $\delta>\max(\eta_1,\alpha\Delta\eta)$ so $\beta\delta>\max(\Delta\eta,\beta\eta_1)$ and we can replace $\beta\delta$ with
$\delta$ as long as $\delta>\max(\Delta\eta,\beta\eta_1)$.
\end{proof}

\section{Relation to the previous two-time formula}\label{secoldformula}

The approach in the present paper can be modified to study the probability
\begin{equation}\label{paA2}
p(a;A)=\mathbb{P}[G(m,n)=a,\,G(M,N)<A],
\end{equation}
under the same scaling (\ref{scaling}).

Let
\begin{equation*}
X'=L^2(\mathbb{R}_-,dx)\oplus L^2(\mathbb{R}_+,dx)\oplus L^2(\{0\},\delta_0)
\end{equation*}
and modify the definition of $S$ and $T$ into
\begin{equation}\label{Sxy2}
S(x,y)=S_1(x,y)+1(x\ge 0)S_2(x,y)-S_3(x,y)1(y<0),
\end{equation}
\begin{equation}\label{Txy2}
T(x,y)=-T_1(x,y)-1(x>0)S_2(x,y)+S_3(x,y)1(y\le 0).
\end{equation}
Define the matrix kernel
\begin{equation}\label{Kuv2}
K_{uv}(x,y)=\begin{pmatrix} R_u(x,y) & R_u(x,y) & R_u(x,y) \\
                                              uR_u(x,y) & uR_u(x,y) & uR_u(x,y)\\
                                              vR_u(x,y) & vR_u(x,y) & vR_u(x,y)
                    \end{pmatrix},
 \end{equation}
where $R_u$ is defined as in (\ref{Ru}) but with $S$ and $T$ given by (\ref{Sxy2}) and (\ref{Txy2}) instead. Then, under (\ref{scaling}),
\begin{equation}\label{limit2}
\lim_{T\to\infty} c_3(t_1T)^{1/3}p(a;A)=\frac{1}{(\pii)^2}\int_{\gamma_r}du\int_{\gamma_r}\frac{dv}{v^2}\det(I+K_{uv})_{X'},
\end{equation}
for any $r>0$. From this formula, it is possible to derive the formula for the two-time distribution given in \cite{JoTt}. It should be possible
to get the formula in \cite{JoTt} also by taking the partial derivative with respect to $\xi_1$ in (\ref{ftt}). We have not been able to carry out that computation.

\bigskip
\noindent
{\bf Acknowledgement}: I thank Jinho Baik for an interesting discussion and correspondence. Also, thanks to Mustazee Rahman for helpful comments on the paper.

\end{document}